\documentclass[reqno,a4paper]{amsart}

\usepackage[utf8]{inputenc}
\usepackage[T1]{fontenc}

\usepackage[pagebackref]{hyperref}
\usepackage{amsfonts}
\usepackage{bbm}
\usepackage{mathtools}
\usepackage{enumerate}
\usepackage[left=2.5cm,right=2.5cm]{geometry}
\allowdisplaybreaks

\newtheorem{theorem}{Theorem}[section]
\newtheorem{lemma}[theorem]{Lemma}
\newtheorem{cor}[theorem]{Corollary} 
\newtheorem{prop}[theorem]{Proposition} 

\theoremstyle{definition}
\newtheorem{defi}[theorem]{Definition}

\theoremstyle{remark}
\newtheorem{example}[theorem]{Example} 
\newtheorem{remark}[theorem]{Remark}

\numberwithin{equation}{section}

\newcommand{\ub}[1]{^{(#1)}}

\newcommand{\ind}[1]{\mathbbm{1}_{#1}} 
\newcommand{\Ind}[1]{\mathbbm{1}_{\{#1\}}} 
\newcommand{\E}{\mathbb{E}} 
\newcommand{\p}{\mathbb{P}} 

\newcommand{\R}{\mathbb{R}} 

\newcommand{\N}{\mathbb{N}}
\newcommand{\X}{\mathbb{X}}
\newcommand{\Y}{\mathbb{Y}}
\newcommand{\cX}{\mathcal{X}}
\newcommand{\cF}{\mathcal{F}}
\newcommand{\un}{^{(n)}}

\DeclareMathOperator{\Med}{Med}	
\DeclareMathOperator{\Var}{Var}	

\author{Rados{\l}aw Adamczak} %
\address[RA]{Institute of Mathematics, University of Warsaw } %
\email{R.Adamczak@mimuw.edu.pl}

\thanks{Research partially supported by the National Science Centre, Poland, grant
no. 2015/18/E/ST1/00214}

\title{On almost sure convergence of random variables with finite chaos decomposition}

\begin{document}

\begin{abstract}
Under mild conditions on a family of independent random variables $(X_n)$ we prove that almost sure convergence of a sequence of tetrahedral polynomial chaoses of uniformly bounded degrees in the variables $(X_n)$ implies the almost sure convergence of their homogeneous parts. This generalizes a recent result due to Poly and Zheng obtained under stronger integrability conditions. In particular for i.i.d. sequences we provide a simple necessary and sufficient condition for this property to hold.

We also discuss similar phenomena for sums of multiple stochastic integrals with respect to Poisson processes, answering a question by Poly and Zheng.
\end{abstract}

\maketitle

\section{Introduction}\label{sec:introduction}
Investigation of real and vector valued multi-linear forms in independent random variables is a classical topic in probability theory closely related to multiple stochastic integration. Such random variables have been thoroughly studied, e.g., in the context of harmonic analysis on the discrete cube, analysis of Boolean functions, geometric theory of Banach spaces, random graphs, concentration of measure, Malliavin calculus or more recently the Malliavin-Stein method. We refer the Reader to the monographs \cite{MR1167198,MR1666908,MR1102015, MR1681553,MR1474726,MR2962301,MR3752640,MR3443800} for extensive exposition of various aspects of the theory.

Recently Poly and Zheng \cite{PolyZheng} have observed that for a large class of sequences $\X = (X_n)_{n\in \N}$ of independent random variables the almost sure convergence of a sequence of sums of tetrahedral (i.e., affine in each variable) multi-linear forms of bounded degrees in the sequence $\X$ can be decomposed into the almost sure convergence of their homogeneous parts. They also proved a counterpart of this result for sums of multiple stochastic integrals with respect to a Gaussian process and posed certain questions concerning similar phenomena for sequences of variables with less regularity than those covered by their theorems, as well as for sums of multiple stochastic integrals with respect to a Poisson process.

The goal of this article is to provide answers to the questions raised by Poly and Zheng and to further study the almost sure convergence of sums of tetrahedral multi-linear forms, also in the vector valued setting. In order to formulate our results in a precise way and to put them in the right perspective let us start with the formulation of the main theorems by Poly and Zheng.

\subsection{Results by Poly and Zheng}

Denote by $\ell_0(\N)^{\odot d}$ the set of all $d$-tensors ($d$-indexed matrices) of the form
$a = (a_{i_1,\ldots,i_d})_{i_1,\ldots,i_d=0}^\infty$, symmetric in their arguments (i.e., $a_{i_1,\ldots,i_d} = a_{i_{\sigma(1)},\ldots,i_{\sigma(d)}}$ for any permutation $\sigma$ of the set $[d] = \{1,\ldots,d\}$), with vanishing diagonals (i.e., such that $a_{i_1,\ldots,i_d} = 0$ whenever $i_k = i_l$ for some $k\neq l$). For $d=0$ we interpret $a \in \ell(\N)^{\otimes d}$ as a single real number $a_\emptyset$ (corresponding to the empty multi-index).

Let $X_0,X_1,X_2,\ldots$ be a family of independent random variables. Assume that $\E X_i = 0$, $\E X_i^2 = 1$ and for some $\delta > 0$, $\sup_i \E |X_{i}|^{2+\delta} < \infty$.

Assume now that $(Z_n)_{1\le n\le \infty}$ is a sequence of random variables of the form\footnote{Note that we include here $n=\infty$}
\begin{displaymath}
  Z_n = \sum_{k=0}^d Z_{n,k},
\end{displaymath}
where
\begin{displaymath}
Z_{n,k} = \sum_{i_1,\ldots,i_k=0}^\infty a^{(n,k)}_{i_1,\ldots,i_k} X_{i_1}\cdots X_{i_k}
\end{displaymath}
for some $a^{(n,k)} \in \ell_0(\N)^{\odot k}$ such that $\sum_{i_1,\ldots,i_k = 1}^\infty |a^{(n,k)}_{i_1,\ldots,i_k}|^2 < \infty$. Here the infinite sums defining $Z_{n,k}$ are understood as a.s. (or $L_2$) limits of sums over $i_1,\ldots,i_k \in \{0,\ldots,n\}$ (their existence follows easily from the martingale convergence theorem). Note that $Z_{n,0}$ are just constants (products over empty index set are interpreted as one).

\medskip

One of the results proved by Poly and Zheng is
\begin{theorem}[Theorem 1.3. in \cite{PolyZheng}]\label{thm:Poly-Zheng-independent}
In the above setting, if $Z_{n}$ converges to $Z_\infty$ a.s. as $n \to \infty$, then for all $k \le d$, $Z_{n,k} \to Z_{\infty,k}$ a.s.
\end{theorem}

In other words the almost sure convergence of sums of tetrahedral multilinear forms of uniformly bounded degrees in the variables $X_i$ decomposes into almost sure convergence of their homogeneous components.

While we postpone the rigorous formulation of our results to subsequent sections, let us announce that we provide a weaker sufficient conditions for this property to hold (see Theorem \ref{thm:general-sequences}), which in particular allows to replace the finiteness of higher moments in Theorem \ref{thm:Poly-Zheng-independent} by uniform square integrability (Corollary \ref{cor:uniformly-integrable}). We also completely characterize i.i.d. sequences with the above property (Theorem \ref{thm:iid}) and extend this phenomenon to the case of multi-linear forms with coefficients from a Banach space (Proposition \ref{prop:Banach}).

Another result from \cite{PolyZheng} is a counterpart for sums of Gaussian multiple stochastic integrals. Since we are not going to use it (we state it only for comparison with the Poissonian case which we will consider in Section \ref{sec:results-Poisson}) we refer, e.g., to the monograph \cite{MR1474726} for the necessary definitions. We remark that the original formulation of the theorem involved rather isonormal Gaussian processes over a separable Hilbert space. To be able to draw analogy with the Poissonian setting, we state it in an equivalent form in terms of Gaussian stochastic measures.

\begin{theorem}[Theorem 1.1. in \cite{PolyZheng}]\label{thm:Poly-Zheng-Gaussian}
  Let  $G$ be a Gaussian stochastic measure on a measurable space $(\mathcal{X},\mathcal{F},\mu)$ and let $I_n$ denote the corresponding $n$-fold Gaussian stochastic integral on
$L_{2,s}(\mathcal{X}^n,\mu^{\otimes n})$ (the space of square integrable functions, symmetric in their arguments). Let $d \in \N$ and consider a sequence $(F_n)_{n=0}^\infty$ of random variables of the form
\begin{displaymath}
  F_n = \E F_n + \sum_{k=1}^d I_k(f_{n,k}),
\end{displaymath}
where $f_{n,k} \in L_{2,s}(\mathcal{X}^k,\mu^{\otimes k})$ and $d \in \N$. If the sequence $F_n$ converges almost surely to a random variable $F$, then $\E F_n \to \E F$ and there exist functions $f_{\infty,k} \in L_{2,s}(\mathcal{X}^k,\mu^{\otimes k})$, such that for all $k\le d$, $I_k(f_{n,k})$ converges almost surely as $n \to \infty$ to $I_k(f_{\infty,k})$.
\end{theorem}

Poly and Zheng ask if an analogous result holds for Poisson multiple stochastic integrals. While we show (see Example \ref{ex:Poisson} below) that this is not the case (even for $d = 1$), we will also prove that under an additional assumption that the converging sequence is majorized by an integrable random variable, one can indeed deduce the almost sure convergence of individual summands from the convergence of the sum (Theorem \ref{thm:Poisson-ui}).

Let us mention in passing that there are many common aspects of the analysis on the Gauss and Poisson space, e.g., they both have the chaos representation property, however the Poisson space lacks many regularity aspects of the Gauss space (e.g., hypercontractivity and related strong concentration properties). Searching for counterparts of Gaussian results in the context of Malliavin calculus, concentration of measure or hypercontractivity is an active area of research (see, e.g., the recent articles \cite{MR3520016,MR3151752,2019arXiv190408211N}). Our result is another example showing that the behaviour of multiple stochastic integrals with respect to the Poisson process resembles to some extent the Gaussian case, however at the cost of introducing some additional assumptions.

\section{Results for independent random variables}\label{sec:results-independent}

We will now present new results for independent random variables, deferring the proofs to further sections. We will start by discussing certain general properties, then we will state the main theorems concerning extensions of Theorem \ref{thm:Poly-Zheng-independent}.

\subsection{Preliminaries}\label{sec:independent-prel}
In order to make the presentation more transparent we need to introduce some additional terminology. Below $\X = (X_i)_{i \in \N}$ is a sequence of independent random variables.

\begin{defi}
For a nonnegative integer $d$ define $Q_d(\X)$ -- the \emph{homogeneous tetrahedral chaos of degree} $d$, as the space of all random variables $Z$, which are limits in probability of a sequence of random variables of the form
\begin{displaymath}
  \sum_{i_1,\ldots,i_d=0}^\infty a_{i_1,\ldots,i_d} X_{i_1}\cdots X_{i_d},
\end{displaymath}
where $a \in \ell_0(\N)^{\otimes d}$ is a $d$-tensor with only finitely many non-zero coefficients.
\end{defi}

\begin{remark}
If the sequence $\X$ consists of i.i.d. Rademacher variables, then $Q_d(\X)$ coincides with the Walsh-Rademacher chaos of order $d$, however if the variables $X_i$ are i.i.d. standard Gaussian, $Q_d$ is distinct from the $d$-th chaos corresponding to the Gaussian Hilbert space spanned by $\X$. Since in this section we discuss only sequences of independent random variables, we believe that this should not lead to misunderstanding. Let us also remark that in general the spaces $Q_d$ may have a non-trivial intersection (see however Proposition \ref{prop:uniqueness} below) and (even if all variables $X_i$ are square integrable) they need not span $L_2(\X)$. Note also that $Q_0(\X)$ is just the space of almost surely constant random variables.
\end{remark}

\begin{defi}\label{def:CDP}
We will say that $\X$ has the \emph{convergence decomposition property} (abbrev. CDP) if for every nonnegative integer $d$ and every sequence $(Z_n)_{1\le n \le \infty}$ of random variables of the form
\begin{align}\label{eq:Z_n}
  Z_n = \sum_{k=0}^d Z_{n,k},
\end{align}
where $Z_{n,k} \in Q_k(\X)$, such that $Z_{n} \to Z_\infty$ a.s. as $n \to \infty$, we have
\begin{align}\label{eq:convergence}
  Z_{n,k} \to Z_{\infty,k}\; \textrm{a.s.}
\end{align}
for all $k = 0,\ldots,d$.
\end{defi}

An obvious necessary condition for the sequence $\X$ to satisfy the CDP is linear independence of the spaces $Q_k(\X)$, i.e. uniqueness of representations of random variables $Z$ as sums of variables from a finite number of spaces $Q_k(\X)$ (if such uniqueness does not hold then the sequence $Z_n = Z$ together with two distinct representations provides a counterexample for the CDP). The following proposition asserts that this minimal condition of uniqueness of the chaos decomposition is in fact also sufficient for the CDP.

\begin{prop}\label{prop:uniqueness} A sequence $\X$ satisfies the CDP if and only if for every $d\in \N$ and every $Y_0,Y_0' \in Q_0(\X), \ldots, Y_d,Y_d' \in Q_d(\X)$, if
\begin{displaymath}
  Y_0 + Y_1+\ldots+ Y_d = Y_0'+Y_1'+\ldots+Y_d'\;\textrm{a.s.},
\end{displaymath}
then for all $k \le d$, $Y_k = Y_k'$ a.s.
\end{prop}

\begin{remark}\label{rem:uniqueness} In fact, as follows from our main technical tool, Lemma \ref{le:reduction-to-d=1} in Section \ref{sec:independent}, if the CDP does not hold, then we can find finite sums  $Z_n = b_n + \sum_{k=0}^{k_n} a\un_k X_k$, where $b_n, a\un_k \in \R$, such that $Z_n \to 0$ almost surely while $b_n \to -1$, $Z_n - b_n \to 1$ a.s. In particular the uniqueness of the decomposition is lost already for $d = 1$.
\end{remark}

The results by Poly and Zheng have been formulated for real valued chaos variables, however it turns out that the CDP automatically extends to an analogous property for polynomial chaoses with coefficients in an arbitrary separable Banach space $(E,\|\cdot\|)$. More precisely, if we define $Q_d(\X,E)$ as the sets of limits in probability of homogeneous tetrahedral polynomials of degree $d$ in $\X$, with coefficients from $E$, then the following result holds.

\begin{prop}\label{prop:Banach} If the sequence $\X$ satisfies the CDP, then for every separable Banach space $(E,\|\cdot\|)$, every non-negative integer $d$  and every sequence $(Z_n)_{1\le n \le \infty}$ of random variables of the form
\begin{displaymath}
  Z_n = \sum_{k=0}^d Z_{n,k},
\end{displaymath}
where $Z_{n,k} \in Q_k(\X,E)$, such that $Z_{n} \to Z_\infty$ almost surely as $n \to \infty$, we have
\begin{displaymath}
  Z_{n,k} \to Z_{\infty,k}
\end{displaymath}
almost surely for all $k = 0,\ldots,d$.
\end{prop}

Let us conclude this section with a comment on the assumed structure of the limiting random variable $Z_\infty$. In the formulation of Theorem \ref{thm:Poly-Zheng-independent} and Proposition \ref{prop:Banach} as well as in Definition \ref{def:CDP} it is assumed that $Z_\infty$ can be also represented as a finite sum of variables from $Q_k(\X)$. The next proposition states that if $\X$ satisfies the CDP, then it is in fact enough to assume just the existence of the limit.

\begin{prop}\label{prop:no-special-form}
Assume that the sequence $\X$ satisfies the CDP and let $(E,\|\cdot\|)$ be a separable Banach space. Consider a sequence of random variables $(Z_n)_{1\le n<\infty}$  as in \eqref{eq:Z_n}, with $Z_{n,k} \in Q_k(\X,E)$. If the sequence $Z_n$ converges in probability to some random variable $Z_\infty$, then there exist unique random variables $Z_{\infty,k}$, $k=0,\ldots,d$ such that $Z_\infty = \sum_{k=0}^d Z_{\infty,k}$ and $Z_k \in Q_k(\X,E)$.
\end{prop}

\subsection{Main results}\label{sec:independent-main}

We will now present the main results for sequences of independent random variables. We will start with a mild sufficient condition for the CDP to hold. Before we formulate it let us note that since the spaces $Q_k(\X)$ do not change when one scales the variables $X_n$ by nonzero factors, there is no loss of generality in assuming that $\X$ is a tight sequence.

\begin{theorem}\label{thm:general-sequences} Let $\X$ be a tight sequence of independent random variables. Assume that for some $\delta > 0$ and all $n \in \N$,
\begin{align}\label{eq:anti-concentration}
\sup_{ x \in \R} \p(|X_n| \in (x-\delta,x+\delta)) \le 1 - \delta.
\end{align}
Assume moreover that there exist $C \ge 0$, $t_0 > 0$ such that for any $t \in (0,t_0]$ and any $n \in \N$,
\begin{align}\label{eq:non-iid}
  \Big|\E X_n \Ind{|X_n| \le \frac{1}{t}}\Big| \le C\Big(\frac{1}{t}\p\Big(|X_n| > \frac{1}{t}\Big) + t \Var\Big(X_n\Ind{|X_n|\le \frac{1}{t}}\Big)\Big).
\end{align}
Then the sequence $\X$ satisfies the CDP.
\end{theorem}

The condition \eqref{eq:non-iid} may seem quite technical, therefore let us now state a corollary to the above theorem, which is a strengthening of Theorem \ref{thm:Poly-Zheng-independent}.

\begin{cor}\label{cor:uniformly-integrable}
If the variables $X_n$ are centered of variance one and the family $\{X_n^2\}_{n\in \N}$ is uniformly integrable, then $\X$ satisfies the CDP.
\end{cor}

\begin{remark} The assumption \eqref{eq:anti-concentration} is an anti-concentration type condition, preventing the random variables $X_n$ from being \emph{too deterministic}. It is not difficult to see that if the variables become too strongly concentrated in points or small intervals away from zero, then the CDP cannot hold, as illustrated by the following example, which answers a question posed by Poly and Zheng \cite{PolyZheng}.
\end{remark}

\begin{example}
Assume that $X_n$ is a centered two point variable of variance one (not necessarily symmetric), i.e., for some $p_n \in (0,1)$,
\begin{displaymath}
\p\Big(X_n = \frac{1-p_n}{\sqrt{p_n(1-p_n)}}\Big) = p_n,\; \p\Big(X_n = \frac{-p_n}{\sqrt{p_n(1-p_n)}}\Big) = 1 - p_n.
\end{displaymath}

Assume that $\limsup_{n \to \infty} p_n = 1$ (the situation when $\liminf_{n \to \infty} p_n = 0$ is completely analogous). In particular there exists an increasing sequence $k_n$ such that
\begin{displaymath}
\sum_{n=0}^\infty (1-p_{k_n}) < \infty.
\end{displaymath}
Set now $Z_{n,0} = 1$, $Z_{n,1} = -\frac{\sqrt{p_{k_n}(1-p_{k_n})}}{1-p_{k_n}} X_{k_n}$, $Z_n = Z_{n,0} + Z_{n,1}$. By the Borel-Cantelli lemma we obtain that with probability one, for sufficiently large $n$, $X_{k_n} =\frac{1-p_{k_n}}{\sqrt{p_{k_n}(1-p_{k_n})}}$ and as a consequence $Z_{n,1} = -1$, $Z_n = 0$. Setting $Z_{\infty,0} = Z_{\infty,1} = 0$ one can see that there is no convergence of of $Z_{n,i}$ to $Z_{\infty,i}$. One can also see that the decomposition of $Z_\infty$ into a sum of variables from $Q_0(\X)$ (constants) and $Q_1(\X)$ is not unique, since $-1 \in Q_1(\X)$ (cf. Proposition \ref{prop:uniqueness} and Remark \ref{rem:uniqueness}). Of course if one insists on representing $Z_\infty$ in the form $c + \sum_{n=0}^\infty a_n X_n$, where $c, a_n \in \R$ then one must have $c = 0$, $a_n \equiv 0$.

On the other hand, if $p_n's$ are separated from zero and one, then it follows from Theorem \ref{thm:Poly-Zheng-independent} that the sequence $\X$ satisfies the CDP.
\end{example}

\medskip \medskip

If the variables $X_n$ are i.i.d., one can show that the condition of Theorem \ref{thm:general-sequences} is in fact necessary for the CDP to hold, i.e., we have the following theorem.

\begin{theorem}\label{thm:iid}
  Assume that the variables $X_n$, $n \in \N$ are i.i.d. Then the sequence $\X$ satisfies the CDP if and only if
there exists $C \ge 0$ and $t_0 > 0$ such that for all $t \in (0,t_0)$,
\begin{align}\label{eq:iff-iid}
  \Big|\E X_0 \Ind{|X_0|\le \frac{1}{t}}\Big| \le C\Big(\frac{1}{t}\p\Big(|X_0| > \frac{1}{t}\Big) + t \Var\Big(X_0\Ind{|X_0|\le \frac{1}{t}}\Big)\Big).
\end{align}
\end{theorem}

\medskip

\begin{remark}
Using the fact that  $\lim_{t \searrow 0} t|\E X_0 \Ind{|X_0| \le t}| = 0$, it is easy to see that \eqref{eq:iff-iid} is satisfied for some $C \ge 0$, $t_0 > 0$ and all $t \in (0,t_0)$ if and only if for some $C_1 \ge 0$, $t_1 > 0$ and all $t \in (0,t_1)$,
\begin{align}\label{eq:iff-iid-1}
\Big|\E X_0 \Ind{|X_0|\le \frac{1}{t}}\Big| \le C_1\Big(\frac{1}{t}\p\Big(|X_0| > \frac{1}{t}\Big) + t \E X_0^2 \Ind{|X_0|\le \frac{1}{t}}\Big).
\end{align}
Furthermore, this is equivalent to the existence of $C_2\ge 0$, such that
\begin{align}\label{eq:iff-iid-2}
\Big|\E X_0 \Ind{|X_0|\le \frac{1}{t}}\Big| \le C_2\Big(\frac{1}{t}\p\Big(|X_0| > \frac{1}{t}\Big) + t \E X_0^2 \Ind{|X_0|\le \frac{1}{t}}\Big)
\end{align}
for all $t > 0$.

Indeed, if $X$ is not equal identically to zero, then for $t$ large enough (say $t > t_2$) and $C_3$ large enough, $|\E X_0 \Ind{|X_0|\le \frac{1}{t}}\Big| \le \frac{1}{t} \le C_3\frac{1}{t}\p(|X_0| > \frac{1}{t})$, while \eqref{eq:iff-iid-2} for $t \in [t_1,t_2]$ can be easily obtained from \eqref{eq:iff-iid-1} for $t < t_1$ (with $C_2$ depending only on $C_1,t_1,t_2$).

Using the Fubini theorem and \eqref{eq:iff-iid-2} one can easily prove that if $Y$ is any random variable independent of $X_0$ and $X_0$ satisfies \eqref{eq:iff-iid}, then so does $X_0Y$ (possibly with different $t_0,C$). This clearly follows from Theorem \ref{thm:iid} but is perhaps somewhat hidden at the level of inequality \eqref{eq:iff-iid}.
\end{remark}

\begin{example} It is clear from the law of large numbers that if $\X$ is an i.i.d. sequence with $X_0$ integrable but not centered, then $\X$ cannot satisfy the CDP. Let us present a sequence violating the CDP with $\E X_0 = 0$. To this end consider $X_0$ satisfying
\begin{displaymath}
  \p\Big(X_0 = \frac{2^n}{n^2}\Big) = \frac{1}{2^{n+1}} \textrm{ for } n = 1,2, \ldots,
\end{displaymath}
and
\begin{displaymath}
  \p\Big(X_0 = -\frac{\pi^2}{6}\Big) = \frac{1}{2}.
\end{displaymath}
Then $\E X_0 = 0$ and for $t = n^2/2^n$ and $n$ large, we obtain
\begin{displaymath}
  \E X_0\Ind{|X_0| \le \frac{1}{t}} = \E X_0\Ind{|X_0| > \frac{1}{t}} = \sum_{k=n+1}^\infty \frac{1}{2k^2} \ge \frac{1}{2(n+1)}.
\end{displaymath}
On the other hand
\begin{displaymath}
  \frac{1}{t}\p\Big(|X_0| > \frac{1}{t}\Big) = \frac{2^n}{n^2} \cdot \frac{1}{2^{n+1}} = \frac{1}{2n^2}
\end{displaymath}
and
\begin{displaymath}
  t \Var\Big(X_0\Ind{|X_0|\le \frac{1}{t}}\Big) \le \frac{n^2}{2^n} \E |X_0|^2\Ind{|X_0| \le 2^n/n^2} = \frac{n^2}{2^n}\Big(\frac{\pi^4}{72} + \sum_{k=1}^n \frac{2^{k-1}}{k^4}\Big) \le \frac{K}{n^2}
\end{displaymath}
for some numerical constant $K$.

This shows that the condition \eqref{eq:iff-iid} is not satisfied and as a consequence $\X$ consisting of i.i.d. copies of $X_0$ does not satisfy the CDP.

\end{example}

\medskip

Our last result concerning independent random variables is the following corollary on reversing the triangle inequality in $L_0$, which should be compared with Lemma \ref{le:anti-triangle} from the Appendix, dealing with $L_p$ spaces for $p\ge 1$. It turns out that in contrast to the $L_p$ case, reversing the triangle inequality at the level of tails requires additional regularity of the distribution of the underlying random variables.

\begin{cor}\label{cor:anti-triangle} Assume that the variables $X_n$, $n \in \N$ are i.i.d. and $X_0$ satisfies \eqref{eq:iff-iid}. Then for any $d \in \N$ there exists a constant $C_d$ such that for any separable Banach space $(E,\|\cdot\|)$, any sequence of random variables $Z_i \in Q_i(\X,E)$, $i = 0,\ldots,d$ and any $t > 0$,
\begin{align}\label{eq:anti-triangle-L0}
\sum_{i=0}^d \p(\|Z_i\| \ge t) \le C_d \p(\|Z_0+Z_1+\ldots+Z_d\| \ge t/C_d).
\end{align}
Moreover, if \eqref{eq:anti-triangle-L0} holds for $E = \R$ and $d=1$, then $X_0$ satisfies \eqref{eq:iff-iid}.
\end{cor}

\section{Multiple Poisson integrals}\label{sec:results-Poisson}

Let us now pass to the Poissonian setting and discuss a counterpart of Theorem \ref{thm:Poly-Zheng-Gaussian}. Since a formal introduction of all the underlying notions is quite lengthy here we will only present the counterexample and the formulation of our theorem, using standard notation from the theory of Poisson processes and Poisson multiple integrals (see, e.g., \cite{MR3791470}), postponing the precise definitions to Section \ref{sec:Poisson}, which will be devoted solely to the Poissonian case.

\begin{example}\label{ex:Poisson}Consider a Poisson process $\eta$ with uniform intensity on the interval $[0,1]$ . Let $f_n = n\ind{[0,\frac{1}{n}]}$ and let $F_n = I_1(f_n) = \int_0^1 f_n d\eta - \int_0^1 f_n dx$ be the compensated Poisson stochastic integral of $f_n$ (in particular $F_n$ is an element of the first Wiener-Poisson chaos, see Section \ref{sec:Poisson} for the definition). Then $F_n$ converges almost surely to $-1$. Since $-1$ is an element of the Poisson chaos of order $0$, we see that the counterpart of Theorem \ref{thm:Poly-Zheng-Gaussian} does not hold even for $d=1$.
\end{example}

On the other hand we have the following result.
\begin{theorem}\label{thm:Poisson-ui}
Let $\eta$ be a Poisson point process on a $\sigma$-finite measurable space $(\cX,\cF,\lambda)$ and for $n \ge 1$ let $I_n$ be the corresponding $n$-fold (compensated) stochastic integral on $L_{2,s}(\mathcal{X}^n,\lambda^{\otimes n})$ (the space of square integrable functions, symmetric in their arguments). Consider  $d \in \N$ and a sequence $(F_n)_{1\le n < \infty}$ of random variables of the form
\begin{displaymath}
  F_n = \E F_n + \sum_{k=1}^d I_k(f_{n,k}),
\end{displaymath}
where $f_{n,k} \in L_{2,s}(\mathcal{X}^k,\lambda^{\otimes k})$. If the sequence $F_n$ converges almost surely to some random variable $F_\infty$, and there exists an integrable random variable $X$ such that for all $n$, $|F_n| \le X$ a.s., then $\E F_n \to \E F_\infty$ and there exist random variables $F_{\infty,k}$, $k=1,\ldots,d$ such that as $n \to \infty$, $I_k(f_{n,k})$ converges almost surely and in $L_1$ to $F_{\infty,k}$. If moreover $(F_n)_{1\le n < \infty}$ is bounded in $L_2$, then there exist functions $f_{\infty,k} \in L_{2,s}(\mathcal{X}^k,\lambda^{\otimes k})$, such that for all $1\le k\le d$, $F_{\infty,k} = I_k(f_{n,k})$.
\end{theorem}

\begin{example}
The assumption that $(F_n)_{n=1}^\infty$ is bounded in $L_2$  cannot be dropped, i.e., if the other assumptions of the theorem are satisfied, but this one is not, it is possible that the sequence $F_n$ converges almost surely to a random variable which is not in $L_2$. To see this it is enough to consider $d=1$, a function $f_\infty \colon \mathcal{X}\to [0,\infty)$, which is integrable but not square integrable and a sequence of functions $f_{n} \in L_2(\mathcal{X},\mu)\cap L_1(\mathcal{X},\mu)$ converging pointwise to $f$ from below. Setting $F_n = I_1(f_n) = \int_\mathcal{X} f_nd\eta - \int_{\mathcal{X}} f_nd\lambda$ one can easily see that $F_n$ converges almost surely to $I_1(f) \notin L_2(\eta)$, moreover $|F_n| \le \int f_\infty d\eta + \int f_\infty d\lambda \in L_1(\eta)$, so $|F_n|$ is indeed dominated by an integrable random variable. On the other hand it is not clear to us whether under the assumption of $L_2$ boundedness or even under a stronger assumption that $F_n$ converge to $F_\infty$ in $L_2$, one can drop the assumption of majorization by an integrable random variable.
\end{example}

\section{Further comments}

\subsection{Overview of the proof} The original proofs of Theorems \ref{thm:Poly-Zheng-independent} and Theorem \ref{thm:Poly-Zheng-Gaussian} due to Poly and Zheng are based on the notion of hypercontractivity, which over the years has proved very useful in analysis of polynomials in random variables. Our approach is based on decoupling inequalities, introduced by McConnell and Taqqu in the 1980s \cite{MR841595} and subsequently developed by many authors, in particular by Kwapie\'n \cite{MR893914} in the case of multilinear forms, with the most general result dealing with $U$-statistics and $U$-processess obtained by de la Pe\~na and Montgomery-Smith \cite{MR1261237} (see Theorem \ref{thm:decoupling} in Appendix \ref{app:decoupling}). This technique reduces the analysis of homogeneous tetrahedral polynomials in a single sequence of independent random variables, to polynomials in multiple copies of this sequence, which are linear in each of the copies (see \cite{MR1652323}, where general decoupling inequalities for $U$-statistics have been used for polynomials in a similar way as in the proof of Lemma \ref{le:reduction-to-d=1} below). This often allows for conditioning and inductive arguments based on the analysis of sequences of independent random variables, which are well understood. The downside of the decoupling approach, when compared with hypercontractivity methods is a typically much  worse dependence of constants in the inequalities on the degree of the polynomial. On the other hand decoupling inequalities are more general as they work for all sequences of independent random variables and also in arbitrary Banach spaces, while hypercontractivity depends heavily on the distribution of the underlying sequence and on the Banach space considered. One can note that from a conditional application of the results by Poly and Zheng it follows that all sequences of symmetric random variables satisfy the CDP, while not all of them satisfy hypercontractive estimates (see \cite{MR1085342} for a characterization in terms of the distribution). This, and the fact that the CDP is a qualitative and not quantitative statement, suggests that in this case decoupling may work more efficiently than hypercontraction. On the other hand we should mention that Corollary \ref{cor:uniformly-integrable} may probably follow by hypercontractive estimates that can be recovered from the proofs in \cite{MR1085342}. Also, Theorem \ref{thm:Poisson-ui} can be proved by means of the Mehler formula for the Poisson process, mimicking the approach Poly and Zheng used in the Gaussian case. We present a sketch of this approach in Section \ref{sec:Poisson}. In terms of notation it is in fact simpler than our main approach based on decoupling, which on the other hand seems to be more easily generalizable to other settings (in particular to more general random measures or $U$-statistics).

\subsection{Organisation of the article}
Section \ref{sec:independent} is devoted to proofs of results for independent random variables. It is split into Subsection \ref{sec:technical lemmas}, where we formulate the main technical lemma and use it to prove Propositions \ref{prop:uniqueness}, \ref{prop:Banach}, \ref{prop:no-special-form}, and Subsection \ref{sec:proofs-indep}, where we present proofs of Theorem \ref{thm:general-sequences}, Corollary \ref{cor:uniformly-integrable}, Theorem \ref{thm:iid} and Corollary \ref{cor:anti-triangle}. Section \ref{sec:Poisson} contains the proof of Theorem \ref{thm:Poisson-ui} on multiple Poisson integrals. In Appendix \ref{app:decoupling} we formulate the main decoupling results that all our proofs are based on, and in Appendix \ref{app:LLN} an elementary proof of Proposition \ref{prop:convergence-to-one} used in Section \ref{sec:proofs-indep}.

\subsection{Acknowledgements} The Author would like to thank prof. Krzysztof Oleszkiewicz for instructive conversations and encouragement to investigate problems discussed in this article.

\section{Proofs of results for independent random variables}\label{sec:independent}

In this section we provide proofs of results concerning sequences of independent random variables, formulated in Section \ref{sec:results-independent}. First we will state a technical lemma and demonstrate abstract propositions, then prove the main theorems.

\subsection{A technical lemma and proofs of results from Section \ref{sec:independent-prel}} \label{sec:technical lemmas}

In what follows by $\stackrel{\p}{\to}$ we denote convergence in probability.

\begin{lemma}\label{le:reduction-to-d=1}
Let $\X = (X_n)_{n=0}^\infty$ be a sequence of independent random variables, satisfying the following implication. For all sequences $a\un = (a\un_0,\ldots,a\un_{k_n}) \in \R^{k_n+1}$, $b_n \in \R$, $n\in \N$, \begin{align}\label{eq:1d-implication}
  \Big( b_n + \sum_{k=0}^{k_n} a\un_k X_k \stackrel{\p, n\to \infty}{\to}0 \Big) \implies \Big(b_n \stackrel{n\to \infty}{\to} 0\Big).
\end{align}
Then for every separable Banach space $(E,\|\cdot\|)$, every non-negative integer $d$  and every sequence $(Z_n)_{1\le n \le \infty}$ of random variables of the form
\begin{displaymath}
  Z_n = \sum_{k=0}^d Z_{n,k},
\end{displaymath}
where $Z_{n,k} \in Q_k(\X,E)$, such that $Z_{n} \to Z_\infty$ almost surely as $n \to \infty$, we have
\begin{displaymath}
  Z_{n,k} \to Z_{\infty,k} \; \textrm{a.s.}
\end{displaymath}
for all $k = 0,\ldots,d$.
\end{lemma}

Before we present the proof of the above lemma, let us make a few comments and describe its basic consequences. In particular we will prove Propositions \ref{prop:Banach}, \ref{prop:uniqueness} and \ref{prop:no-special-form}.

\medskip

Let us start with the following remark, which will be used in the proof of Lemma \ref{le:reduction-to-d=1}.

\begin{remark}\label{re:1d-prob-as}
Consider the implication \eqref{eq:1d-implication} with convergence in probability replaced by the almost sure convergence. It is easy to see that this formally weaker property of the sequence $\X$ is in fact equivalent to \eqref{eq:1d-implication}. Indeed, assume that such a weaker version holds and consider any $a\un, b_n$ such that
\begin{displaymath}
  b_n + \sum_{k=0}^{k_n} a\un_k X_k \stackrel{\p}{\to} 0.
\end{displaymath}
For every increasing sequence $n_m$ of nonnegative integers we can find a subsequence $n_{m_l}$ such that
\begin{displaymath}
  b_{n_{m_l}} + \sum_{k=0}^{k_{n_{m_l}}} a\un_k X_k \stackrel{a.s., l\to \infty}{\to} 0,
\end{displaymath}
which implies that $b_{n_{m_l}} \to 0$ as $l \to \infty$. Therefore, $b_n \to 0$ as $n\to \infty$, which proves \eqref{eq:1d-implication}.
\end{remark}

The above remark and Lemma \ref{le:reduction-to-d=1} immediately implies the following corollary.
\begin{cor}\label{cor:equivalence} A sequence $\X$ of independent random variables satisfies the CDP if and only if it satisfies the implication \eqref{eq:1d-implication}.
\end{cor}

Having the above corollary we can easily prove Propositions \ref{prop:Banach} and \ref{prop:uniqueness}.

\begin{proof}[Proof of Proposition \ref{prop:Banach}] If $\X$ satisfies the CDP, then by Corollary \ref{cor:equivalence} it satisfies the implication \eqref{eq:1d-implication}. The assertion of the proposition follows thus by Lemma \ref{le:reduction-to-d=1}.
\end{proof}

\begin{proof}[Proof of Proposition \ref{prop:uniqueness}]
The necessity of the uniqueness of the decomposition is obvious. To show that it is also sufficient for the CDP, note that by Lemma \ref{le:reduction-to-d=1} if this property does not hold, then we can find a sequence of linear forms in the variables $X_i$ converging in probability to 1. Thus $1 \in Q_1(\X)$. Since obviously $1 \in Q_0(\X)$, this shows that there is no uniqueness of the decomposition for $d=1$.
\end{proof}

Finally let us demonstrate Proposition \ref{prop:no-special-form}.

\begin{proof}[Proof of Proposition \ref{prop:no-special-form}]
It is enough to prove the existence of the variables $Z_{\infty, k}$. Consider thus any strictly increasing sequences $l_n, m_n$ of integers. Since $Z_n$ converges in probability, the difference $S_n := Z_{l_n} - Z_{m_n}$ converges in probability to zero. Thus from an arbitrary subsequence of $S_n$ we can select a further subsequence along which the almost sure convergence holds. Using the CDP, we obtain that along this subsequence also the homogeneous parts of $S_n$ tend almost surely to zero. Thus from every subsequence of $S_{n,k} := Z_{l_n,k} - Z_{m_n,k}$ one can select a further subsequence converging almost surely to zero, which implies that $S_{n,k}$ converges to zero in probability. But, as the sequences $l_n, m_n$ were arbitrary, this implies that the Cauchy condition for convergence in probability is satisfied, and so by the completeness of $L_0(E)$, $Z_{n,k}$ converges in probability to some random variable $Z_{\infty,k}$. Clearly we have then $Z_\infty = \sum_{k=0}^d Z_{\infty,k}$.
\end{proof}

We will now pass to the proof of Lemma \ref{le:reduction-to-d=1}.

\begin{proof}[Proof of Lemma \ref{le:reduction-to-d=1}]
We will use the notation as in Definition \ref{def:CDP}. Clearly it is enough to consider the case $Z_{\infty,k} = 0$ for all $k \le d$. Also, we can assume that $Z_{n,k}$ are
multilinear tetrahedral forms in a finite number of variables $X_i$, i.e.
\begin{align}\label{eq:finite-sums}
  Z_{n,k} = \sum_{i_1,\ldots,i_k=0}^{\infty} a^{(n,k)}_{i_1,\ldots,i_k} X_{i_1}\cdots X_{i_k}
\end{align}
where $a^{(n,k)} \in \ell_0^{\otimes n}(\N)$ and there exist $m_{n,k} < \infty$ such that $a^{(n,k)}_{i_1,\ldots,i_k} = 0$ if $\max(i_1,\ldots,i_k) > m_{n,k}$.

Indeed, by the Borel-Cantelli lemma we can find $\widetilde{Z}_{n,k}$, $0\le k\le d, n\ge 0$, being such tetrahedral forms, such that with probability one for all $k \le d$, $\widetilde{Z}_{n,k} - Z_{n,k} \to 0$ as $n\to \infty$. In particular $\sum_{k=0}^d \widetilde{Z}_{n,k} \stackrel{a.s.}{\to} 0$ and  for all $k$, $Z_{n,k} \stackrel{a.s.}{\to 0}$ iff
$\widetilde{Z}_{n,k} \stackrel{a.s.}{\to} 0$. For the purpose of the proof we can thus assume that $\widetilde{Z}_{n,k} = Z_{n,k}$.

We will now  prove by induction on $d\ge 1$ that for \emph{any sequence} $\X$, satisfying \eqref{eq:1d-implication}, if $Z_{n,k}$, $k\le d$, are as in \eqref{eq:finite-sums} and $Z_n = \sum_{k=0}^d Z_{n,k} \stackrel{a.s.}{\to} 0$, then for all $k \le d$, $Z_{n,k} \stackrel{a.s.}{\to} 0$.

\medskip

\paragraph{\bf The base of induction: $d = 1$}
In this case we have $Z_n = Z_{n,0} + Z_{n,1}$, where $Z_{n,0} \in E$ is deterministic and $Z_{n,1} = \sum_{k=0}^{k_n} a^{(n,1)}_k X_k$ for some $a^{(n,1)}_k \in E$.
By the Hahn-Banach theorem there exist norm one linear functionals  $\varphi_n$  on $E$ such that $\varphi_n(Z_{n,0}) = \|Z_{n,0}\|$. If $Z_n \stackrel{a.s.}{\to} 0$, then
$\varphi_n(Z_n) = \|Z_{n,0}\| + \sum_{k=0}^{k_n} \varphi_n(a^{(n,1)}_k) X_k \stackrel{a.s.}{\to} 0$. By assumpion \eqref{eq:1d-implication} this implies that $\|Z_{n,0}\| \to 0$, which clearly yields $Z_{n,1} \stackrel{a.s.}{\to} 0$.

\paragraph{\bf The induction step} Let us assume that the induction hypothesis holds for all numbers smaller than $d$. Note that by the convergence $Z_n \stackrel{a.s.}{\to} 0$, we have for arbitrary $\varepsilon > 0$,
\begin{align}\label{eq:as-finite-dimensional}
\lim_{n\to \infty} \sup_{m \ge n} \p(\max_{n \le l \le m} \|Z_{l}\| > \varepsilon) = 0.
\end{align}

For $l \in \N$ define the functions $h^{(l)}_{i_1,\ldots,i_d} \colon \R^d \to E$ by the formula
\begin{displaymath}
  h^{(l)}_{i_1,\ldots,i_d}(x_1,\ldots,x_d) = \sum_{k=0}^d \frac{(d-k)!}{d!} \frac{(N-d)!}{(N-k)!}\sum_{1\le r_1 \neq \ldots\neq r_{k} \le d} a^{(l,k)}_{i_{r_1},\ldots,i_{r_k}}x_{r_1}\cdots x_{r_k}
\end{displaymath}
and note that the random vector $(Z_n,\ldots,Z_m)$ can be written as
\begin{displaymath}
  \sum_{1\le i_1\neq i_2 \neq \ldots \neq i_d \le N} \Big(h^{(l)}_{i_1,\ldots,i_d}(X_{i_1}\ldots,X_{i_d})\Big)_{l=n}^m
\end{displaymath}
where $N = \max_{n\le l \le m} \max_{0\le k \le d} m_{l,k}$ (here and in what follows the notation $i_1\neq \ldots\neq i_d$ denotes the condition that the indices $i_j$ are \emph{pairwise} distinct).

Let now $\X^{(i)} = (X^{(i)}_n)_{n=0}^\infty$, $i \in [d]$ be i.i.d. copies of the sequence $\X$ and define $Z^{dec}_l$, the decoupled version of $Z_l$ as
\begin{align*}
  Z^{dec}_l &:= \sum_{1\le i_1\neq i_2 \neq \ldots \neq i_d \le N} h^{(l)}_{i_1,\ldots,i_d}(X^{(1)}_{i_1}\ldots,X^{(d)}_{i_d})\\
  & = \sum_{k=0}^d \frac{1}{\binom{d}{k}} \sum_{1\le r_1<\ldots< r_k\le d} \sum_{i_1,\ldots,i_k = 1}^\infty a^{(l,k)}_{i_1,\ldots,i_k} X^{(r_1)}_{i_1}\cdots X^{(r_k)}_{i_k},
\end{align*}
where the equality follows from the symmetry of the tensors $a^{(l,k)}$, and the fact that they have vanishing diagonals and finite support.

For every permutation $\pi$ of the set $[d]$ we have $h^{(l)}_{i_{\pi(1)},\ldots,i_{\pi(d)}}(x_{\pi(1)},\ldots,x_{\pi(d)}) = h^{(l)}_{i_1,\ldots,i_d}(x_{i_1},\ldots,x_{i_d})$, so by Theorem
\ref{thm:decoupling} from the Appendix, applied to the space $F = E^{\{n,\ldots,m\}}$ equipped with the norm $\|(x_n,\ldots,x_m)\| = \max_{n\le l \le m}\|x_i\|$, we have for every $\varepsilon > 0$,
\begin{displaymath}
  \frac{1}{C_d} \p(\max_{n \le l \le m} \|Z_{l}\| \ge C_d \varepsilon ) \le \p(\max_{n\le l \le m} \|Z_{l}^{dec}\| \ge \varepsilon) \le C_d\p(\max_{n \le l \le m} \|Z_{l}\| \ge  \varepsilon/C_d )
\end{displaymath}
Combining this with \eqref{eq:as-finite-dimensional} we obtain that $Z^{(dec)}_n \stackrel{a.s.}{\to} 0$ as $n \to \infty$.

Consider a sequence $\Y = (Y_n)_{n=0}^\infty$ defined as
\begin{displaymath}
  Y_{kd+r} = X^{(r+1)}_k
\end{displaymath}
for $k\in \N$, $r\in \{0,\ldots,d-1\}$, and any sequences $a\un = (a\un_0,\ldots,a\un_{k_n}) \in \R^{k_n+1}$, $b_n \in \R$, $n\in \N$, such that
$b_n + \sum_{k=0}^{k_n} a\un_k Y_k \to 0$ almost surely. Using the Fubibi theorem and applying successively \eqref{eq:1d-implication} to $X^{(r)}$ ($r \in [d]$) conditionally on $X^{(l)}$, $l\in [d]\setminus \{r\}$, we easily obtain that $b_n \to 0$ as $n \to \infty$. Taking into account Remark \ref{re:1d-prob-as} we can infer that $\Y$ satisfies the implication \eqref{eq:1d-implication}.

Now, for fixed $r \in [d]$, applying this implication to $(X^{(r)}_n)_{n\in \N}$ and $Z^{dec}_n$, conditionally on $\{X_i^{(r)}\}_{n\in \N,r\in [d]\setminus \{r\}}$ we obtain via the Fubini theorem, that
\begin{displaymath}
  \sum_{k=0}^{d-1} \frac{1}{\binom{d}{k}} \sum_{{1\le r_1<\ldots< r_k\le d}\atop{\forall_i r_i \neq r}} \sum_{i_1,\ldots,i_k = 1}^\infty a^{(n,k)}_{i_1,\ldots,i_k} X^{(r_1)}_{i_1}\cdots X^{(r_k)}_{i_k} \stackrel{a.s.}{\to} 0.
\end{displaymath}
The induction hypothesis applied to $\Y$ implies now that for each $r \in [d]$ and each $k\le d-1$
\begin{displaymath}
  \sum_{{1\le r_1<\ldots< r_k\le d}\atop{\forall_i r_i \neq r}} \sum_{i_1,\ldots,i_k = 1}^\infty a^{(n,k)}_{i_1,\ldots,i_k} X^{(r_1)}_{i_1}\cdots X^{(r_k)}_{i_k} \stackrel{a.s.}{\to} 0
\end{displaymath}
(we use here that every tetrahedral homogeneous polynomial can represented in the form \eqref{eq:finite-sums}).

Now we get
\begin{displaymath}
  \sum_{{1\le r_1<\ldots< r_k\le d}} \sum_{i_1,\ldots,i_k = 1}^\infty a^{(n,k)}_{i_1,\ldots,i_k} X^{(r_1)}_{i_1}\cdots X^{(r_k)}_{i_k}  = \frac{1}{d-k}\sum_{r = 1}^d \sum_{{1\le r_1<\ldots< r_k\le d}\atop{\forall_i r_i \neq r}} \sum_{i_1,\ldots,i_k = 1}^\infty a^{(n,k)}_{i_1,\ldots,i_k} X^{(r_1)}_{i_1}\cdots X^{(r_k)}_{i_k} \stackrel{a.s.}{\to} 0,
\end{displaymath}
for all $k \le d-1$ and as a consequence
\begin{displaymath}
 Z^{dec}_{n,d} := \sum_{i_1,\ldots,i_d=1}^\infty  a^{(n,d)}_{i_1,\ldots,i_d}X^{(1)}_{i_1}\cdots X^{(d)}_{i_d} \stackrel{a.s.}{\to 0}.
\end{displaymath}
Applying now again the decoupling inequality, we conclude that for all $\varepsilon > 0$,
\begin{displaymath}
  \limsup_{n\to \infty} \sup_{m \ge n} \p(\max_{n \le l \le m} \|Z_{l,d}\| > \varepsilon) \le C\lim_{n\to \infty} \sup_{m \ge n} \p(\max_{n \le l \le m} \|Z^{dec}_{l,d}\| > \varepsilon/C) = 0,
\end{displaymath}
i.e., $Z_{n,d} \stackrel{a.s.}{\to} 0$. From this we obtain $\sum_{k=0}^{d-1} Z_{n,k} \stackrel{a.s.}{\to 0}$, which by another application of the induction hypothesis implies that $Z_{n,k} \stackrel{a.s.}{\to} 0$ also for all $k < d$, and ends the induction step.
\end{proof}

\subsection{Proofs of results from Section \ref{sec:independent-main}}\label{sec:proofs-indep}

We will use the following proposition, characterizing convergence in probability to a constant for row sums of a triangular array  of independent random variables. Let us remark that with some not difficult but technical calculations  it can be obtained from a much deeper result, namely \cite[Chapter IV, Theorem 3]{MR0388499}, characterizing weak convergence of such sums to an arbitrary infinitely divisible distribution. However, to make the presentation more self contained and elementary, we provide a direct proof in Appendix \ref{app:LLN}.

\begin{prop}\label{prop:convergence-to-one}Let $X_{n,k}$, $n\in \N, k\in \{0,\ldots,k_n\}$ be a trianglar array of random variables such that for each $n$, $X_{n,0},\ldots,X_{n,k_n}$ are independent.
Assume that for all $\varepsilon > 0$,
\begin{align}\label{eq:asymptotic-smallness}
  \max_{0\le k\le k_n} \p(|X_{n,k}|\ge \varepsilon) \to 0
\end{align}
as $n \to \infty$. Let $\tau$ be an arbitrary positive number. Then $\sum_{k=0}^{k_n} X_{n,k}$ converges in probability to $1$ if and only if
\begin{itemize}
\item[(i)]
\begin{align}\label{eq:1st-condition}
\sum_{k=0}^{k_n} \E X_{n,k}\Ind{|X_{n,k}| \le \tau} \to 1
\end{align}

and

\item[(ii)]
\begin{align}\label{eq:2nd-condition}
\sum_{k=0}^{k_n} \Big(\p(|X_{n,k}| > \tau) + \Var\Big(X_{n,k}\Ind{|X_{n,k}|\le \tau}\Big)\Big) \to 0
\end{align}
as $n \to \infty$.
\end{itemize}

\end{prop}

We are now ready for the proof of Theorem \ref{thm:general-sequences}.

\begin{proof}[Proof of Theorem \ref{thm:general-sequences}]
By Lemma \ref{le:reduction-to-d=1} it is enough to verify that under the assumptions of Theorem  \ref{thm:general-sequences} the implication \eqref{eq:1d-implication} holds. We will proceed by contradiction. Assume thus that there are sequences $a\un = (a\un_0,\ldots,a\un_{k_n}) \in \R^{k_n+1}$, $b_n \in \R$, $n\in \N$, such that $b_n + \sum_{k=0}^{k_n} a\un_k X_k \stackrel{\p, n\to \infty}{\to}0$ but $b_n$ does not converge to 0. By passing to a subsequence we can further assume that $b_n$'s are separated from zero. Dividing by $b_n$ and setting $t_{n,k} = -a_{n,k}/b_{n}$ we thus obtain a sequence $t_n = (t_{n,0},\ldots,t_{n,k_n})$ such that
\begin{displaymath}
  \sum_{k=0}^{k_n} t_{n,k}X_k \stackrel{\p}{\to} 1.
\end{displaymath}
Let $\X' = (X_n')_{n=0}^\infty$ be a independent copy of $\X$. We have
\begin{displaymath}
  \sum_{k=0}^{k_n} t_{n,k}(X_k - X_k') \stackrel{\p}{\to} 0.
\end{displaymath}
By assumption \eqref{eq:anti-concentration} and the Fubini Theorem we obtain for all $k$,
\begin{displaymath}
  \p(|X_k - X_k'| \ge \delta) \ge \delta.
\end{displaymath}
On the other hand, by symmetry of $X_k-X_k'$, for any $\varepsilon > 0$,
\begin{displaymath}
  \p(|t_{n,k}|\cdot|X_k - X_k'| \ge \varepsilon) \le 2\p\Big(\Big|\sum_{k=0}^{k_n} t_{n,k}(X_k - X_k')\Big| \ge \varepsilon\Big) \to 0,
\end{displaymath}
which shows that $t_{n,k}$ converge with $n$ to zero, uniformly in $k\in \N$. Together with tightness this implies that the triangular array given by $X_{n,k} = t_{n,k}X_k$ satisfies \eqref{eq:asymptotic-smallness}.
As a consequence, by Proposition \ref{prop:convergence-to-one} we obtain
\begin{displaymath}
  A_n := \sum_{{1\le k\le k_n}\atop{t_{n,k}\neq 0}}  t_{n,k} \E X_k\Ind{|X_k|\le \frac{1}{|t_{kn}|}} \to 1
\end{displaymath}
and
\begin{displaymath}
  B_n := \sum_{{1\le k \le k_n}\atop{t_{n,k}\neq 0}} \Big(\p\Big(|X_k| > \frac{1}{|t_{n,k}|}\Big) + t_{n,k}^2\Var\Big(X_k\Ind{|X_{n,k}\le \frac{1}{|t_{n,k}|}}\Big)\Big) \to 0,
\end{displaymath}
which is however impossible, since by \eqref{eq:non-iid}, for $n$ large enough, we have $|A_n| \le C B_n$.
This ends the proof of the theorem.
\end{proof}

\begin{remark}
Let us note that using the characterization of the weak law of large numbers for triangular arrays given in \cite[Chapter IX, Theorem 1]{MR0388499}, one can obtain the following version of Theorem \ref{thm:general-sequences}, without assuming the anti-concentration condition \eqref{eq:anti-concentration}. Since the replacement for condition
\eqref{eq:non-iid} in this result is more involved, and Theorem \ref{thm:general-sequences} is sufficient for the main situations of interest (in particular for the case of centered square uniformly integrable variables of variance one covered by Corollary \ref{cor:uniformly-integrable}, and for the i.i.d. case treated in Theorem \ref{thm:iid}) we will not provide a detailed proof.
\end{remark}

\begin{theorem}
  Assume that there exists a constant $C\ge 0$ such that for all $t > 0$, and all $n \in \N$,
  \begin{multline*}
    |\Med(X_n) + \E (X_n-\Med(X_n))\Ind{|X_n-\Med(X_n)| \le \frac{1}{t}}| \\
    \le C\Big(\frac{1}{t}\p\Big(|X_n| > \frac{1}{t}\Big) + t \E |X_n - \Med(X_n)|^2\Ind{|X_n - \Med(X_n)| \le \frac{1}{t}}\Big).
  \end{multline*}
Then $\X$ satisfies the CDP.
\end{theorem}

Let us now prove Corollary \ref{cor:uniformly-integrable}.

\begin{proof}[Proof of Corollary \ref{cor:uniformly-integrable}]
  Let us first prove the condition \eqref{eq:non-iid}. Let $t_0$ be such that for all $n\in \N$, $\E X_n^2\Ind{|X_n| \le \frac{1}{t_0}} \ge 1/2$.
  Using the mean zero assumption, for $t < t_0$ we can estimate
  \begin{multline*}
  |\E X_n \Ind{|X_n| \le \frac{1}{t}}| = |\E X_n \Ind{|X_n| > \frac{1}{t}}| \le t\E X_n^2 = t \\
  \le 2t \E X_n^2\Ind{|X_n| \le \frac{1}{t}} = 2t \Var\Big(X_n\Ind{|X_n| \le \frac{1}{t}}\Big) + 2t\Big(\E X_n\Ind{|X_n| > \frac{1}{t}}\Big)^2.
  \end{multline*}
Now, by the Schwarz inequality, the second term on the right-hand side above is bounded by
\begin{displaymath}
  2t \E X_n^2 \p\Big(|X_n| > \frac{1}{t}\Big) \le 2t_0^2 \frac{1}{t}\p\Big(|X_n| > \frac{1}{t}\Big),
\end{displaymath}
which shows that \eqref{eq:non-iid} holds with $C = \max(2,2t_0^2)$.

Tightness of the sequence $\X$ follows from uniform integrability, so to finish the proof it remains to demonstrate the condition \eqref{eq:anti-concentration}. This will follow by uniform integrability and a Paley-Zygmund type argument.

By the de la Vall\'ee Poussin theorem, there exists a convex, nondecreasing function $\varphi\colon [0,\infty) \to [0,\infty)$ such that $\varphi(0) = 0$, $\lim_{x\to \infty} \varphi(x)/x = \infty$ and $M:= \sup_{n\in \N} \E \varphi(|X_n|^2) < \infty$. Consider any $x \in [-2,2]$.
By convexity
\begin{displaymath}
  \E\varphi\Big(\frac{(X_n-x)^2}{4}\Big) \le \frac{1}{2}(\E \varphi(|X_n|^2) + \varphi(4)) \le \frac{1}{2}(M + \varphi(4))
\end{displaymath}
Therefore, again by convexity, there exists $K$ such that for all $n\in \N$ and $x \in [-2,2]$,
\begin{displaymath}
  \E\varphi\Big(\frac{(X_n-x)^2}{K}\Big) \le \frac{1}{4}.
\end{displaymath}

On the other hand $\E (X_n- x)^2 = \E X_n^2 + x^2 \ge 1$.
Denoting by $\varphi^\ast$ the Legendre transform of $\varphi$, given by the formula $\varphi^\ast(x) = \sup_{y\ge 0}(xy - \varphi(y))$, we can estimate
\begin{displaymath}
\frac{3}{4} \le \E (X_n-x)^2\Ind{|X_n-x| > \frac{1}{2}} \le \E\varphi\Big(\frac{(X_n-x)^2}{K}\Big) + \varphi^\ast(K)\p\Big(|X_n-x| \ge \frac{1}{2}\Big),
\end{displaymath}
which together with the definition of $K$ yields
\begin{displaymath}
  \p\Big(|X_n-x| \ge \frac{1}{2}\Big) \ge \frac{1}{2\varphi^\ast(K)}
\end{displaymath}
for all $x\in [-2,2]$.
For $|x|> 2$, by Chebyshev's inequality we have $\p(|X-x| \le 1/2) \le \p(|X| \ge 3/2) \le \frac{4}{9}$.

Combining the last two estimates we obtain \eqref{eq:anti-concentration} with $\delta = 2^{-1}\min(1,1/\varphi^\ast(K))$, which ends the proof of the corollary.
\end{proof}

We will conclude this section by proving the characterizations of the CDP (Theorem \ref{thm:iid}) and the corresponding reverse triangle inequality (Corollary \ref{cor:anti-triangle}) in the i.i.d. case.

\begin{proof}[Proof of Theorem \ref{thm:iid}]
One can easily check that the equivalence between the CDP and condition \eqref{eq:iff-iid} holds in the case of almost surely constant variable $X_0$ (both conditions are satisfied if and only if $X_0$ vanishes almost surely), therefore from now on we will assume that $X_0$ is not deterministic. We will first prove that \eqref{eq:iff-iid} implies the CDP. To this end we will use Theorem \ref{thm:general-sequences}. The condition \eqref{eq:non-iid} in the i.i.d. case clearly reduces to \eqref{eq:iff-iid}, tightness of $\X$ is obvious, and the condition \eqref{eq:anti-concentration} follows easily from the assumption that $X_0$ is not deterministic. Indeed, for any pair of sequences $x_n \in \R$ and $\delta_n \to 0$, such that $\p(X_0 \in (x_n-\delta_n,x_n+\delta_n)) \ge 1 - \delta_n$, the sequence $x_n$ must be bounded, and thus passing to a convergent subsequence we would obtain that $X_0$ is deterministic. Thus, as all the assumptions of Theorem \ref{thm:general-sequences} hold, we can conclude that $\X$ satisfies the CDP.

Let us now prove the converse implication. Assume that \eqref{eq:iff-iid} is not satisfied. Thus there exists a sequence of positive numbers $t_n$, such that $t_n \to 0$ and
\begin{displaymath}
  t_n \Big|\E X_0 \Ind{|X_0|\le \frac{1}{t_n}}\Big| > n \Big(\p\Big(|X_0| > \frac{1}{t_n}\Big) + t_n^2 \Var\Big(X_0\Ind{|X_0|\le \frac{1}{t_n}}\Big)\Big).
\end{displaymath}
Set $a_n = t_n$ if $\E X_0\Ind{|X_0|\le \frac{1}{t_n}} > 0$ and $a_n = -t_n$ otherwise. Define moreover
\begin{displaymath}
k_n = \Big\lfloor \Big(t_n \Big|\E X_0 \Ind{|X_0|\le \frac{1}{t_n}}\Big|\Big)^{-1}\Big\rfloor - 1.
\end{displaymath}
Note that by the Lebesgue dominated convergence theorem
\begin{align}\label{eq:convergence-to-zero}
t_n \Big|\E X_0 \Ind{|X_0|\le \frac{1}{t_n}}\Big| \to 0
\end{align}
and so $k_n \to \infty$.

Now consider the sequence $Z_n = \sum_{k=0}^{k_n} a_n X_k = \sum_{k=0}^{k_n} X_{n,k}$, where $X_{n,k} = a_n X_k$. Since $a_n \to 0$ and $X_n$ have the same distribution, the condition \eqref{eq:asymptotic-smallness} is satisfied. Using the definition of $k_n$ and \eqref{eq:convergence-to-zero} we get
\begin{displaymath}
  \sum_{k=0}^{k_n} \E X_{n,k}\Ind{|X_{n,k}| \le 1} = (k_n +1)t_n \Big|\E X_0 \Ind{|X_0|\le \frac{1}{t_n}}\Big| \to 1
\end{displaymath}
as $n \to \infty$, which yields \eqref{eq:1st-condition} of Proposition \ref{prop:convergence-to-one}.

Moreover,
\begin{align*}
  \sum_{k=0}^{k_n} \Big( \p(|X_{n,k}| > 1) + \Var(X_{n,k}\Ind{|X_{n,k}| \le 1})\Big) &= (k_n+1) \Big(\p\Big(|X_0| > \frac{1}{t_n}\Big) + t_n^2 \Var\Big(X_0\Ind{|X_0|\le \frac{1}{t_n}}\Big)\Big)\\
  &\le \frac{\p\Big(|X_0| > \frac{1}{t_n}\Big) + t_n^2 \Var\Big(X_0\Ind{|X_0|\le \frac{1}{t_n}}\Big)}{t_n \Big|\E X_0 \Ind{|X_0|\le \frac{1}{t_n}}\Big|} < \frac{1}{n}
\end{align*}
and so \eqref{eq:2nd-condition} is also satisfied.

Thus, by Proposition \ref{prop:convergence-to-one} we obtain that  $\sum_{k=0}^{k_n} a_n X_n \to 1$ in probability. Passing to a subsequence we can upgrade this to the almost sure convergence, which implies that the CDP cannot hold.
\end{proof}

\begin{proof}[Proof of Corollary \ref{cor:anti-triangle}]
To prove the first part of the corollary we will proceed by contradiction, constructing a sequence of polynomials with coefficients in $c_0$, which violate the assertion of Proposition \ref{prop:Banach}. Let us thus assume that \eqref{eq:iff-iid} holds but \eqref{eq:anti-triangle-L0} is violated. Then there exist $d$, $k \le d$, a sequence of Banach spaces $E_n$, $t_n > 0$ and $Z_{n,i} \in Q_i(\X,E_n)$ ($i \le d$), such that
\begin{displaymath}
  \p(\|Z_{n,k}\| \ge 2t_n) > 4n^2 \p\Big(\|Z_{n,0}+\ldots+Z_{n,d}\| \ge \frac{t_n}{2n^2}\Big),
\end{displaymath}
(for simplicity we will denote all the norms appearing in the proof by $\|\cdot\|$). Scaling $Z_{n,k}$ if necessary, we can assume that $t_n = 1$. By approximation we obtain homogeneous tetrahedral forms (in particular depending on a finite number of variables) $Z'_{n,i}$ of degree $i$ ($i\le d$), such that
\begin{displaymath}
  \p(\|Z'_{n,k}\| \ge 3/2) > 4n^2 \p\Big(\|Z'_{n,0}+\ldots+Z'_{n,d}\| \ge \frac{2}{3n^2}\Big).
\end{displaymath}
Passing to subspaces spanned by coefficients of $Z'_{n,i}$ we may further assume that all spaces $E_n$ are finite dimensional, which by a standard embedding gives a sequence $N_n$ of positive integers, and tetrahedral forms $Z_{n,k}''$  with  values in $\ell_\infty^{N_n}$ such that
\begin{displaymath}
  \p(\|Z''_{n,k}\| \ge 1) > n^2 \p\Big(\|Z''_{n,0}+\ldots+Z''_{n,d}\| \ge \frac{1}{n^2}\Big).
\end{displaymath}

Let now $m_n = \lceil 1/\p(\|Z''_{n,k}\| \ge 1) \rceil$. Since $Z''_{n,i}$ depend only on finitely many variables $X_n$, using the sequence $\X$ we can construct i.i.d. copies
$(Z''_{n,1}(j),\ldots,Z''_{n,d}(j))$, $j = 1,\ldots,m_n$ of the vectors $(Z''_{n,1},\ldots,Z''_{n,d})$. Then $\widehat{Z}_{n,i} := (Z''_{n,i}(j))_{j=1}^{m_n}$ may be considered a tetrahedral homogeneous polynomial of degree $i$ with coefficients in $\ell_{\infty}^{N_nm_n}$ embedded in $c_0$ in a natural way. Recall also the following elementary inequality for independent random variables $\xi_i$:
\begin{displaymath}
  \frac{1}{2}\min\Big(\sum_j \p(\xi_j > t), 1\Big) \le \p(\max_j \xi_j > t) \le \sum_{j} \p(\xi_j > t).
\end{displaymath}
Using this inequality together with independence over $j = 1,\ldots,m_n$ we obtain
\begin{align}\label{eq:contradiction-yo}
\p(\|\widehat{Z}_{n,k}\| \ge 1) = \p(\max_{j\le m_n} \|Z_{n,k}''(j)\|\ge 1) \ge \frac{1}{2} \min\Big(m_n \p(\|Z_{n,k}''\| \ge 1),1\Big) = 1/2
\end{align}
and

\begin{align*}
  \p\Big(\|\widehat{Z}_{n,0} + \ldots+\widehat{Z}_{n,d}\| \ge \frac{1}{n^2}\Big) &= \p\Big(\max_{j\le m_n} \|Z_{n,0}''(j) + \ldots+Z_{n,d}''(j)\| \ge \frac{1}{n^2}\Big)\\
  & \le m_n \p\Big(\|Z''_{n,0} + \ldots+Z''_{n,d}\| \ge \frac{1}{n^2}\Big) \\
  &\le m_n \frac{1}{n^2} \p(\|Z''_{n,k}\| \ge 1) \le \frac{2}{n^2}.
\end{align*}
Thus, by the Borel-Cantelli lemma, the sequence $\widehat{Z}_{n} = \widehat{Z}_{n,0} + \ldots+\widehat{Z}_{n,d}$ of $c_0$ valued tetrahedral polynomials converges almost surely to $0$, while by \eqref{eq:contradiction-yo}, $\widehat{Z}_{n,k}$ does not. By Proposition \ref{prop:Banach} this shows that $\X$ does not have the CDP, which by Theorem \ref{thm:iid} contradicts \eqref{eq:iff-iid} and finishes the proof of the first part of the corollary.

As for the second part, if \eqref{eq:anti-triangle-L0} is satisfied for $d=1$ and $E = \R$, then clearly the implication \eqref{eq:1d-implication} holds and thus, by Lemma \ref{le:reduction-to-d=1}, $\X$ satisfies the CDP. By Theorem \ref{thm:iid} this implies that \eqref{eq:iff-iid} is satisfied.
\end{proof}

\section{Proofs of results for Poisson stochastic integrals}\label{sec:Poisson}
In this section we will prove Theorem \ref{thm:Poisson-ui}. The basic proof we will provide is again based on decoupling inequalities. After completing the argument we will also present a sketch of the proof based on Mehler's formula for the Poisson process. We choose to focus on the decoupling proof since it is a variation on the approach we used for independent random variables and also it seems that its adaptation to more general situations (i.e., other random measures) is more straightforward than in the case of Mehler's formula argument.

Let us start by recalling the basic definitions of multiple stochastic integrals with respect to the Poisson process. Clearly we are not able to provide here a complete exposition, so we will just present the basic formulas and constructions necessary for carrying out the proof, and refer the reader to the monograph \cite{MR3791470} for details.

The multiple Wiener-It\^{o} integral $I_n\colon L_{2,s}(\mathcal{X}^n,\lambda^{\otimes n})\to L_2(\Omega,\p)$ is defined first for integrable $f$ with an explicit formula \eqref{eq:Wiener-Ito-integrable} below and then uniquely extended to the space $L_{2,s}(\mathcal{X}^n,\lambda^{\otimes n})$, by a standard density argument, in such a way that $I_n/\sqrt{n!}$ is an isometric embedding. For $f\colon \mathcal{X}^n\to \R$, integrable (not necessarily symmetric or square integrable) one defines
\begin{align}\label{eq:Wiener-Ito-integrable}
I_n(f) = \sum_{J\subset [n]} (-1)^{n-|J|} \int_{\mathcal{X}^{|J|}}f(x_1,\ldots,x_n)\eta^{(|J|)}(dx_J)\lambda^{n-|J|}(dx_{J^c})
\end{align}
where $\eta^{(n)}$ is the $m$-th factorial measure on $\mathcal{X}^m$, defined inductively by $\eta^{(1)}=\eta$,
\begin{displaymath}
  \eta^{(m+1)}(\cdot) = \int_{\mathcal{X}^m}\Big(\int_{\mathcal{X}}\Ind{(x_1,x_2,\ldots,x_{m+1})\in\cdot}\eta(dx_{m+1}) - \sum_{i=1}^m \Ind{(x_1,x_2,\ldots,x_m,x_i)\in\cdot}\Big)\eta^{(m)}(d(x_1,\ldots,x_m)).
\end{displaymath}
If $\eta$ is a proper point process, i.e., if $\eta$ can be represented as a countable sum of Dirac's deltas $\eta = \sum_{i=1}^\kappa \delta_{X_i}$ for some $\N\cup\{\infty\}$-valued random variable $\kappa$ and $\mathcal{X}$-valued random variables $X_i$, then
\begin{displaymath}
  \mu^{(m)} = \sum_{i_1,\ldots,i_m=1}^\kappa \Ind{i_1\neq \ldots\neq i_m}\delta_{(x_{i_1},\ldots,x_{i_k})}.
\end{displaymath}

In particular, if $B_1\ldots,B_n \subset \mathcal{X}$ are pairwise disjoint and $B = B_1\times\ldots\times B_n$, then $I_n(\ind{B})=\prod_{i=1}^n (\eta(B_i)-\lambda(B_i))$.
One also proves that $I_n(g)$ and $I_m(f)$ are uncorrelated for $n\neq m$. The subspace of $L_2(\Omega)$ consisting of all $m$-fold stochastic integrals of square integrable symmetric functions in $m$ variables is called the \emph{$m$-th Wiener-Poisson chaos}. The chaos representation property asserts that these spaces form an orthogonal decomposition of the space of square integrable random variables measurable with respect to $\eta$, which we will denote by $L_2(\eta)$ (see \eqref{eq:CR} below for an explicit formula).

\begin{proof}[Proof of Theorem \ref{thm:Poisson-ui}]
For the proof of Theorem \ref{thm:Poisson-ui} it will be convenient to assume that the measure $\lambda$ is non-atomic. We can do it without loss of generality, since we can replace $\mathcal{X}$ with $\mathcal{X}\times (0,1)$, $\lambda$ with $\lambda \otimes Leb$ (where $Leb$ is the Lebesgue measure on the interval) and $f_{n,k}$ by $f_{n,k}\circ \pi^k$, where $\pi^k$ is the natural projection from $(\mathcal{X}\times(0,1))^k$ onto $\mathcal{X}^k$. One can then check that the assumptions of Theorem \ref{thm:Poisson-ui} remain unchanged as well as the joint distribution of all the stochastic integrals involved.

Given a square $\lambda^{\otimes k}$-integrable symmetric function $f\colon \mathcal{X}^k \to \R$, by $\sigma$-finiteness of $\lambda$, we can approximate it in $L_2$ by a bounded symmetric function $g$ supported on a set $K^k$ with $\lambda(K) < \infty$. Then, using the well-known Darboux property for non-atomic measures, we can split $K$ into nested measurable partitions $\mathcal{A}_n = \{A_{n,1},\ldots,A_{n,2^n}\}$ of sets of measure $\lambda(K)/2^n$. Using the martingale convergence theorem we can approximate $g$ by functions constant on the sets $A_{n,i_1}\times \ldots\times A_{n,i_k}$. Since the total measure of sets of this form with $i_j = i_l$ for some $j\neq l$ is at most $\lambda(K)^k k(k-1)2^{-nk}2^{n(k-1)} = o(1)$ as $n \to \infty$, it follows that we can approximate $f$ in $L_2$ by functions of the form
\begin{align}\label{eq:tetrahedral-step-function}
  h = \sum_{i_1,\ldots,i_k = 1}^N a_{i_1,\ldots,i_k}\ind{A_{i_1}\times \ldots \times A_{i_k}},
\end{align}
where the sets $A_1,\ldots,A_N$ are pairwise disjoint subsets of $\mathcal{X}$ with $\lambda(A_i) < \infty$, the coefficients $a_{i_1,\ldots,i_k}$ are symmetric and vanish if $i_l = i_m$ for some $l\neq m$. Note also that if we have a finite family of functions of this form (perhaps with different $k$'s and $N$'s), we can always find their representations with the same sets $A_1,\ldots,A_N$ (first one enlarges the corresponding sequences of sets to have the same union, then one takes all possible intersections).

In the setting of Theorem \ref{thm:Poisson-ui}, we can thus find functions $g_{n,k} \in L_{2,s}(\mathcal{X}^k,\lambda^k)$, $k=1,\ldots,n$, such that as $n\to \infty$,
\begin{displaymath}
  \sum_{n=0}^\infty \sum_{k=1}^d \|I_n(f_{n,k}) - I_n(g_{n,k})\|_2 < \infty.
\end{displaymath}
Define $Z_n = \E F_n + \sum_{k=1}^d I_k(g_{n,k})$. It follows from the Borel-Cantelli lemma that for each $k$, $I_n(f_{n,k}) - I_n(g_{n,k})$ tends to zero almost surely. In particular $Z_n$ converges almost surely to $F_\infty$. Moreover,
\begin{displaymath}
\E \sup_{n \in \N} |Z_n| \le \E \sup_{n\in N} |F_n| + \E\sup_{n\in \N} |Z_n - F_n| \le \E X + \sum_{n=0}^\infty\sum_{k=1}^d \|I_n(f_{n,k}) - I_n(g_{n,k})\|_2 < \infty.
\end{displaymath}
Therefore it is enough to prove the almost sure convergence of $Z_{n,k} := I_n(g_{n,k})$. To this end we will closely follow the strategy used in the proof of Lemma \ref{le:reduction-to-d=1}.

Assume that $g_{n,k}$ is of the form
\begin{displaymath}
  g_{n,k} = \sum_{i_1,\ldots,i_k=1}^{N_{n}} a^{(n,k)}_{i_1,\ldots,i_k}\ind{A_{n,i_1}\times \ldots\times A_{n,i_k}},
\end{displaymath}
where the sets $A_{n,i}$ are pairwise disjoint and of finite measure $\lambda$ and coefficients $a^{(n,k)}_{i_1,\ldots,i_k}$ are symmetric and with vanishing diagonals (as explained before we can assume that the family of sets $A_{n,1},\ldots,A_{n,N_n}$ does not depend on $k$).
Note that
\begin{displaymath}
Z_{n,k} = I_k(g_{n,k}) = \sum_{i_1,\ldots,i_k=1}^{N_{n}} a^{(n,k)}_{i_1,\ldots,i_k}\prod_{j=1}^k(\eta(A_{n,i_j}) - \lambda(A_{n,i_j}))
\end{displaymath}
Let $\eta_1, \ldots,\eta_d$ be independent copies of the Poisson process $\eta$ and define the decoupled version of $Z_{n}$ with the formula
\begin{displaymath}
  Z_{n}^{dec} = \E F_n + \sum_{k=1}^d \frac{1}{\binom{d}{k}} \sum_{1\le r_1<\ldots< r_k\le d} \sum_{i_1,\ldots,i_k = 1}^{N_n} a^{(n,k)}_{i_1,\ldots,i_k}
  \prod_{j=1}^k(\eta_{r_j}(A_{n,i_j}) - \lambda(A_{n,i_j})).
\end{displaymath}

The almost sure convergence of $Z_n$ can be written as the following Cauchy type condition
\begin{displaymath}
  \lim_{n \to \infty} \sup_{m>n} \p(\sup_{n \le l \le m} |Z_n - Z_l| \ge \varepsilon) = 0
\end{displaymath}
for all $\varepsilon > 0$, while the majorization by an integrable random variable as
\begin{displaymath}
  \lim_{m\to \infty} \E \sup_{0\le l \le m} |Z_l| < \infty.
\end{displaymath}

Fix $m$ and recall that there exists $M$ and pairwise disjoint sets of finite measure $\lambda$, $B_1,\ldots,B_{M}$ together with symmetric coefficients $b^{(l,k)}_{i_1,\ldots,i_k}$, vanishing on diagonals, such that for every $l \le m$,
\begin{displaymath}
  g_{l,k} = \sum_{i_1,\ldots,i_k=1}^{M} b^{(l,k)}_{i_1,\ldots,i_k}\ind{B_{i_1}\times \ldots\times B_{i_k}},
\end{displaymath}
so that
\begin{displaymath}
  Z_{l,k} = \sum_{i_1,\ldots,i_k=1}^{M} b^{(l,k)}_{i_1,\ldots,i_k}\prod_{j=1}^k\Big(\eta(B_{i_j}) - \lambda(B_{i_j})\Big)
\end{displaymath}
(to simplify the notation we suppress the dependence of $M$ and the sets $B_i$ on $m$).

Thus, setting $X_i = \eta(B_i) - \lambda(B_i)$, we get for $l \le m$,
\begin{displaymath}
  Z_l = \sum_{1\le i_1\neq \ldots\neq i_d \le M} h^{(l)}_{i_1,\ldots,i_d} (X_{i_1},\ldots,X_{i_d}),
\end{displaymath}
where
\begin{displaymath}
  h^{(l)}_{i_1,\ldots,i_d}(x_1,\ldots,x_d) = \frac{(M-d)!}{M!}\E F_n + \sum_{k=1}^d \frac{(d-k)!}{d!}\frac{(M-d)!}{(M-k)!}\sum_{1\le r_1\neq \ldots\neq r_k \le d} b^{(l,k)}_{i_{r_1},\ldots,i_{r_k}}x_{r_1}\cdots x_{r_k}.
\end{displaymath}
Denote $X^{(j)}_i = \eta_j(B_i) - \lambda(B_i)$. Using the additivity of $\eta_j$ and $\lambda$, one can check that for $l \le m$,
\begin{displaymath}
  Z_{l}^{dec} = \sum_{1\le i_1\neq \ldots\neq i_d \le M} h^{(l)}_{i_1,\ldots,i_d} (X^{(1)}_{i_1},\ldots,X^{(d)}_{i_d})
\end{displaymath}
and hence applying the decoupling inequalities of Theorem \ref{thm:decoupling} to the spaces $\ell_\infty(\{n,n+1,\ldots,m\})$ and $\ell_\infty(\{0,1,\ldots,m\})$ and functions
\begin{align*}
  H_{i_1,\ldots,i_d}(x_1,\ldots,x_d) &= (h^{(l)}_{i_1,\ldots,i_d}(x_1,\ldots,x_d) - h^{(n)}_{i_1,\ldots,i_d}(x_1,\ldots,x_d))_{l=n}^m\\
  G_{i_1,\ldots,i_d}(x_1,\ldots,x_d) &= (h^{(l)}_{i_1,\ldots,i_d}(x_1,\ldots,x_d))_{l=0}^m
\end{align*}
respectively, we obtain
\begin{displaymath}
  \lim_{n \to \infty} \sup_{m>n} \p(\sup_{n \le l \le m} |Z^{dec}_n - Z^{dec}_l| \ge \varepsilon) = 0
\end{displaymath}
for all $\varepsilon > 0$, and
\begin{displaymath}
  \lim_{m\to \infty} \E \sup_{0\le l \le m} |Z^{dec}_l| < \infty,
\end{displaymath}
i.e., $Z^{dec}_n$ converges almost surely and is dominated by an integrable function. By Fubini Theorem, if we fix $s_1<\ldots< s_k \in [d]$, then with probability one $Z^{dec}_n$ converges almost surely with respect to $\{\eta_{i}\colon i \in [d]\setminus\{s_1,\ldots,s_k\}\}$ and is almost surely dominated by some integrable random variable. Thus with probability one it converges in $L_1(\eta_{i}\colon i \in [d]\setminus\{s_1,\ldots,s_k\})$, and in particular $\E(Z^{dec}_n|\eta_{s_1},\ldots,\eta_{s_k})$ converges almost surely
for every choice of $s_1,\ldots,s_k$. But
\begin{align}\label{eq:conditioned-decoupled-integral}
  \E(Z^{dec}_n|\eta_{s_1},\ldots,\eta_{s_k}) = \E F_n + \sum_{l=1}^k \frac{1}{\binom{d}{l}} \sum_{{1\le r_1<\ldots< r_l\le d}\atop{r_1,\ldots,r_l} \subset \{s_1,\ldots,s_k\}} \sum_{i_1,\ldots,i_l = 1}^{N_n} a^{(n,l)}_{i_1,\ldots,i_l}
  \prod_{j=1}^l(\eta_{r_j}(A_{n,i_j}) - \lambda(A_{n,i_j}))
\end{align}

From this, by induction one easily proves that $\E F_n$ is convergent and for any $1\le k \le d$, the sequence
\begin{displaymath}
  Z_{n,k}^{dec} = \sum_{i_1,\ldots,i_k = 1}^{N_n} a^{(n,k)}_{i_1,\ldots,i_k}
  \prod_{j=1}^k(\eta_{j}(A_{n,i_j}) - \lambda(A_{n,i_j}))
\end{displaymath}
converges almost surely.
Indeed, taking $\{s_1,\ldots,s_d\} = \emptyset$, we obtain convergence of $\E F_n$. Now assuming that $\E F_n$ converges and $Z_{n,l}^{dec}$ for $1 \le l < k$ converge almost surely,
by equidistribution of $\eta_i$ we obtain that for any $l < k$
\begin{displaymath}
\sum_{{1\le r_1<\ldots< r_l\le d}\atop{r_1,\ldots,r_l} \subset \{1,\ldots,k\}} \sum_{i_1,\ldots,i_l = 1}^{N_n} a^{(n,l)}_{i_1,\ldots,i_l}  \prod_{j=1}^l(\eta_{r_j}(A_{n,i_j}) - \lambda(A_{n,i_j}))
\end{displaymath}
converges almost surely, which combined with the almost sure convergence of $\E(Z^{dec}_n|\eta_1,\ldots,\eta_k)$ and \eqref{eq:conditioned-decoupled-integral} yields the almost sure convergence of $Z_{n,k}^{dec}$.

Now, using the decoupling inequalities in the opposite direction than before (we skip the definition of the corresponding functions $h$, which in this case is easier, since we deal with homogeneous polynomials), we obtain that the sequence $Z_{n,k}$ converges almost surely for each $k \le d$. By Lemma \ref{le:anti-triangle} we obtain that
\begin{displaymath}
  \E \sup_{n\in \N} |Z_{n,k}| \le C \E \sup_{n\in \N} |Z_n| < \infty,
\end{displaymath}
so we also have convergence in $L_1$ (note that Lemma \ref{le:anti-triangle} could be recovered from the above decoupling arguments, in fact this is the way it was proved in \cite{MR3052405}, but we prefer to rely on the abstract formulation, so as not to further complicate the above elementary but notationally unpleasant arguments).
We have thus established that the variables $Z_{n,k}$ converge almost surely and in $L_1$ to some random variables $F_{\infty,k}$ and (as explained at the beginning of the argument) it follows that the same convergence holds for $F_{n,k}$. In particular we have $F_\infty = \E F_\infty + \sum_{k=1}^d F_{\infty,k}$.

It remains to prove that if $(F_n)_{n=0}^\infty$ is bounded in $L_2$, then $F_{\infty,k}$ can be expressed as $k$-fold stochastic integral of a square integrable symmetric function. Note that by orthogonality, for each $k \le d$, $(F_{n,k})_{n=1}^\infty$ is bounded in $L_2$. Thus one can select a subsequence $(I_{k}(f_{n_l,k}))_{l=1}^\infty$, which converges weakly in $L_2$ to some random variable $\widetilde{F}_{\infty,k}$. Since the $k$-th chaos is a closed linear subspace of $L_2$, it follows that $\widetilde{F}_{\infty,k} = I_k(f_{\infty,k})$ for some $f_{\infty,k} \in L_{2,s}(\mathcal{X}^k,\lambda^{\otimes k})$. Moreover by the convergence of $F_{n,k}$ to $F_{\infty,k}$ in $L_1$, we obtain that for every measurable set $A$, $\E F_{\infty,k}\ind{A} = \lim_{n\to \infty} \E F_{n,k}\ind{A} = \E \widetilde{F}_{\infty,k}\ind{A}$, which shows that $F_{\infty,k} = \widetilde{F}_{\infty,k}$ almost surely and ends the proof of the theorem.
\end{proof}

\begin{remark}Let us note that variants of the above argument can be repeated to prove the almost sure convergence in more general situations, e.g., for square integrable random fields for which one defines multiple stochastic integrals by the $L_2$ theory, for tetrahedral polynomial chaos based on sequences of independent random variables (as investigated in the previous section) or for $U$-statistics, as in all these settings we can apply the general decoupling inequality in a similar manner.
\end{remark}

\begin{proof}[A sketch of an alternate proof of Theorem \ref{thm:Poisson-ui}]
The argument we will present is based on Mehler's formula for the Poisson process and is a direct counterpart of the proof of Theorem \ref{thm:Poly-Zheng-Gaussian} due to Poly and Zheng. We will not provide all the technical details, and refer to the monograph \cite{MR3791470} for the details. Recall that for any function $F$ on the set of integer valued measures on $\mathcal{X}$ (denote it by $N(\mathcal{X})$) and $\mu \in N(\mathcal{X})$ and any $x \in \mathcal{X}$ we define $D_x F(\mu) = D^1_x F(\mu)= F(\mu + \delta_x) - F(\mu)$ and inductively $D_{x_1,\ldots,x_n}F(\mu) = D_{x_1}^1D_{x_2,\ldots,x_{n}}^{n-1}F(\mu)$. We  also set $D^0 F = F$. For $F$ as above we also define the symmetric functions $T_n F\colon \mathcal{X}^n \to \R$ as $T_nF(x_1,\ldots,x_n) = \E D_{x_1,\ldots,x_n} F(\eta)$. For $F \in L_2(\eta)$ we then have the chaos representation, namely the equality
\begin{align}\label{eq:CR}
  F(\eta) = \sum_{n=0}^\infty \frac{1}{n!}I_n(T_n F),
\end{align}
with the series converging in $L_2(\eta)$. If $\eta$ is proper, i.e., it can be almost surely represented as a sum of Dirac's deltas, $\eta = \sum_{k=1}^\kappa \delta_{X_n}$ $(\kappa \le \infty$), we also define the $t$-trimming of $\eta$ ($t \in [0,1]$) as
\begin{displaymath}
  \eta_t = \sum_{k=1}^\kappa \Ind{U_n \le t} \delta_{X_n},
\end{displaymath}
where $U_1,U_2,\ldots$ are independent random variables distributed uniformly on $[0,1]$, independent of $\eta$.

Finally one defines a family of operators $P_t$, $t \in [0,1]$ on $L_2(\eta)$ with the formula
\begin{align}\label{eq:Pt-def}
  P_t f= \E f(\eta_t + \eta_{1-t}')|\eta),
\end{align}
where $\eta'_{1-t}$ is a Poisson process with intensity $(1-t)\lambda$, independent of the pair $(\eta,\eta_t)$. Mehler's formula (\cite[Chapter 20]{MR3520016}, \cite{MR3791470}) asserts then that for any $f \in L_2(\eta)$ and $t \in [0,1]$,
\begin{align}\label{eq:Mehler}
  D^n_{x_1,\ldots,x_n} (P_t F) = t^n P_t D_{x_1,\ldots,x_n}F,\; \textrm{and}\; \E D^n_{x_1,\ldots,x_n} (P_t F) = t^n \E D_{x_1,\ldots,x_n}F.
\end{align}

In the setting of Theorem \ref{thm:Poisson-ui}, we can assume without loss of generality that $\eta$ is proper (since we can always find  proper Poisson process with the same distribution as $\eta$, cf. \cite[Corollary 3.7]{MR3791470}).

We can assume that $\eta$ and $\eta'_{1-t}$ are defined on a product probability space $\Omega = \Omega_\eta\times \Omega_{\eta'}$ with measure $\p_\eta \otimes \p_{\eta'}$ and that they depend respectively only on the first and second coordinate.
Since $\widetilde{\eta} = \eta_t + \eta_{1-t}'$ has the same distribution as $\eta$, if we define $\widetilde{F}_n = \E F_n + \sum_{k=1}^d \widetilde{I}_k(f_{n,k})$, where $\widetilde{I}_k$ is the $k$-fold Wiener-It\^o integral with respect to $\widetilde{\eta}$, then $\widetilde{F}_n$ also converges in distribution and $\E \sup_n |\widetilde{F}_n|$ is integrable. Thus, by the Fubini theorem, it follows that $\p_\eta$-almost surely, the sequence $\widetilde{F}_n$  converges almost surely with respect to $\p_{\eta'}$ and is uniformly integrable. In particular, using the definition \eqref{eq:Pt-def} we obtain that $P_t F_n = \int \widetilde{F}_n d\p_{\eta'}$ converges almost surely as $t \to \infty$.

On the other hand \eqref{eq:Mehler} and the chaos representation property \eqref{eq:CR} imply that
\begin{displaymath}
  P_t F_n = \E F_n + \sum_{k=1}^d t^k I_k(f_{n,k}).
\end{displaymath}
Using the fact that the right-hand side above converges almost surely for sufficiently many $t \in [0,1]$, we obtain that $I_k(f_{n,k})$ converges almost surely for each $k \le d$.
\end{proof}

\appendix
\section{Decoupling and related inequalities}\label{app:decoupling}

In this section we gather basic facts concerning decoupling inequalities for $U$-statistics that are used throughout the article.

Let us start with the by now classical decoupling inequality due to de la Pe\~{n}a and Montgomery-Smith.

\begin{theorem}{\cite[Theorem 1]{MR1261237}}\label{thm:decoupling} Let $d$ be a positive integer and for $n \ge d$ let $(X_i)_{i=1}^n$ be a sequence of independent random variables with values in a measurable space $(S,\mathcal{S})$ and let $(X\ub{j}_i)_{i=1}^n$ $j= 1,\ldots,d$ be $d$ independent copies of this sequence. Let $E$ be a separable Banach space and for each $(i_1,\ldots,i_d) \in [n]^d$ with pairwise distinct coordinates let $h_{i_1,\ldots,i_d} \colon S^d \to E$ be a measurable function. There exists a numerical constant $C_d$, depending only on $d$ such that for all $t > 0$,
\begin{displaymath}
\p\Big(\Big\|\sum_{1\le i_1\neq \ldots\neq i_d\le n} h_{i_1,\ldots,i_d} (X_{i_1},\ldots,X_{i_d})\Big\|> t\Big)
\le C_d \p\Big(\Big\|\sum_{1\le i_1\neq \ldots\neq i_d\le n} h_{i_1,\ldots,i_d}(X\ub{1}_{i_1},\ldots,X\ub{d}_{i_d})\Big\|> t/C_d\Big).
\end{displaymath}
As a consequence for all $p \ge 1$,
\begin{displaymath}
\Big\|\sum_{1\le i_1\neq \ldots\neq i_d\le n} h_{i_1,\ldots,i_d}  (X_{i_1},\ldots,X_{i_d})\Big\|_p \le C_d' \Big\|\sum_{1\le i_1\neq \ldots\neq i_d\le n} h_{i_1,\ldots,i_d} (X\ub{1}_{i_1},\ldots,X\ub{d}_{i_d})\Big\|_p,
\end{displaymath}
where $C_d'$ is another numerical constant depending only on $d$.

If moreover the functions $h_{i_1,\ldots,i_d}$ are symmetric in the sense that, for all $x_1,\ldots,x_d \in S$ and all permutations $\pi\colon [d]\to [d]$,
$h_{i_1,\ldots,i_d}(x_1,\ldots,x_d) = h_{i_{\pi_1},\ldots,i_{\pi_d}}(x_{\pi_1},\ldots,x_{\pi_d})$,
then for all $t > 0$,
\begin{displaymath}
\p\Big(\Big\|\sum_{1\le i_1\neq \ldots\neq i_d\le n} h_{i_1,\ldots,i_d} (X\ub{1}_{i_1},\ldots,X\ub{d}_{i_d})\Big\|> t\Big)
\le \widetilde{C}_d \p\Big(\Big\|\sum_{1\le i_1\neq \ldots\neq i_d\le n} h_{i_1,\ldots,i_d} (X_{i_1},\ldots,X_{i_d})\Big\|> t/\widetilde{C}_d\Big)
\end{displaymath}
where $\widetilde{C}_d$ is a constant depending only on $d$.
As a consequence for some numerical constant $\widetilde{C}_d'$, depending only on $d$, and all $p \ge 1$,
\begin{displaymath}
\Big\|\sum_{1\le i_1\neq \ldots\neq i_d\le n} h_{i_1,\ldots,i_d} (X\ub{1}_{i_1},\ldots,X\ub{d}_{i_d})\Big\|_p
\le \widetilde{C}_d'\Big\|\sum_{1\le i_1\neq \ldots\neq i_d\le n} h_{i_1,\ldots,i_d} (X_{i_1},\ldots,X_{i_d})\Big\|_p.
\end{displaymath}

\end{theorem}

Another result used in our proofs is the following reverse triangle inequality for tetrahedral chaoses, obtained for the first time by Kwapie\'n \cite{MR893914} in the symmetric setting, which easily gives the general case (see also \cite{MR3052405}, where an alternate proof in the general case, based on Theorem \ref{thm:decoupling} is presented). We remark that this lemma can be also obtained by methods used by Poly and Zheng in their proof of Theorem \ref{thm:Poly-Zheng-independent}.

\begin{lemma}\label{le:anti-triangle}
For $j = 0,1,\ldots,d$ let $(a_{i_1,\ldots,i_j}^{j})_{1\le i_1,\ldots,i_j \le n}$ be a $k$-indexed symmetric array
of real numbers (or more generally elements of some normed space), such that $a_{i_1,\ldots,i_j}^{j} = 0$ if $i_k = i_l$ for some $1 \le k < l \le j$ (for $j=0$ we have just a single number $a_\emptyset^{{0}}$). Let $X_1,\ldots,X_n$ be independent mean zero random variables. Then there exists a constant $C_d \in (0,\infty)$, depending only on $d$, such that for all $p \ge 1$,
\begin{displaymath}
\sum_{j=0}^d \Big\|\sum_{i_1,\ldots,i_j=1}^n a_{i_1,\ldots,i_j}^{j}X_{i_1}\cdots X_{i_j}\Big\|_p \le
C_d \Big\|\sum_{j=0}^d\sum_{i_1,\ldots,i_j=1}^n a_{i_1,\ldots,i_j}^{j}X_{i_1}\cdots X_{i_j}\Big\|_p.
\end{displaymath}
\end{lemma}

\section{Proof of Proposition \ref{prop:convergence-to-one}}\label{app:LLN}
We will now prove the characterization of the convergence in probability to one, given in Proposition \ref{prop:convergence-to-one}.
\begin{proof}
  Assume first that conditions (i), (ii) are satisfied. By (ii) we get that $\p(\max_{i\le k_n} |X_{n,i}| > \tau) \to 0$ and as a consequence $\sum_{k=0}^{k_n} X_{k,n}\Ind{|X_{n,k}| > \tau}$ converges in probability to zero. On the other hand, by (i), (ii) and Chebyshev's inequality, $\sum_{k=0}^{k_n} X_{n,k}\Ind{|X_{n,k}|\le \tau}$ converges in probability to one, which ends the proof of the first implication (note that we did not use the asymptotic smallness condition \eqref{eq:asymptotic-smallness}).

  Assume now that $\sum_{k=0}^{k_n} X_{n,k}$ converges in probability to one. Denote $\X_n = (X_{n,k})_{k=0}^{k_n}$ and let $\X' = (X_{n,k}')_{k=0}^{k_n}$ be an independent copy of $\X$. We have $\sum_{k=0}^{k_n}(X_{n,k} - X_{n,k}') \to 0$ in probability. Since $X_{n,k} - X_{n,k}'$ is symmetric we also have $S_n := \sum_{k=0}^{k_n} \varepsilon_k |X_{n,k} - X_{n,k}'|$ where $\varepsilon_k, k \in \N$ are i.i.d. Rademacher variables independent of $(X_{n,k}), (X_{nk}')$. Consider the event
  \begin{displaymath}
  A_n= \{\max_{k \le k_n} |X_{n,k}| > \tau\} = \bigcup_{k=0}^{k_n} A_{n,k},
  \end{displaymath}
  where $A_{n,k} = \{\forall_{0\le i< k} |X_{i,n}| \le \tau, |X_{n,k}| > \tau\}$.
   Note that by independence and \eqref{eq:asymptotic-smallness} for large $n$, on $A_{n,k}$, $\p(|X_{n,k} - X_{n,k}'|\ge \tau/2|\X) \ge 1/2$ . Moreover, by symmetry of the Rademacher variables $\p(|S_n| \ge |X_{n,k} - X_{n,k}'||\X,\X') \ge 1/2$. Therefore we get
   \begin{displaymath}
     \p(|S_n| \ge \tau/2) \ge \sum_{k=0}^{k_n} \p(\{|S_n| \ge \tau/2\}\cap A_{n,k}) \ge \frac{1}{4}\sum_{k=0}^{k_n} \p(A_{n,k}) = \p(A_n)/4.
   \end{displaymath}
   As a consequence $\p(A_n) \to 0$ as $n \to \infty$.
  A standard estimate
  \begin{displaymath}
  \p(A_n) \ge \frac{1}{2}\min\Big(\sum_{k=0}^{k_n} \p(|X_{n,k}|> \tau),1\Big)
  \end{displaymath}
    shows that
   \begin{align}\label{eq:sum-of-rpobabilities}
    \sum_{k=0}^{k_n} \p(|X_{n,k}| > \tau) \to 0
    \end{align}
    as $n \to \infty$.

   Define now $Z_{n,k} = (X_{n,k}\Ind{|X_{n,k}|\le \tau} - X_{n,k}'\Ind{|X_{n,k}'|\le \tau})$ and $\widetilde{S}_n = \sum_{k=0}^{n} Z_{n,k}$. We have $\widetilde{S}_{n} \to 0$ in probability.
Moreover, $\E Z_{n,k} = 0$ and so by independence,
  \begin{displaymath}
    \E \widetilde{S}_n^ 4 = \sum_{k=0}^{k_n} \E Z_{n,k}^4 + 3\E \sum_{1\le i \neq j \le k_n} \E Z_{n,i}^2 \E Z_{n,j}^2 \le 4 \tau^2 \E \widetilde{S}_n^2 + 3(\E \widetilde{S}_n^2)^2.
    \end{displaymath}
   By the Paley-Zygmund inequality (see, e.g., \cite[Corollary 3.3.2]{MR1666908}),
   \begin{displaymath}
     \p\Big(|\widetilde{S}_n| \ge \frac{1}{2}(\E \widetilde{S}_n^2)^{1/2}\Big) \ge \frac{9}{16} \frac{(\E \widetilde{S}_n^2)^2}{\E \widetilde{S}_n^4}
     \ge \frac{9}{16}\frac{(\E \widetilde{S}_n^2)^2}{4 \tau^2 \E \widetilde{S}_n^2 + 3(\E \widetilde{S}_n^2)^2}.
   \end{displaymath}
   This shows that $\E \widetilde{S}_n^2 \to 0$ as $n \to \infty$ (since otherwise the right hand side above would be separated from zero along a subsequence). But
   $\E \widetilde{S}_n^2 = 2 \sum_{k=0}^{k_n} \Var(X_{n,k}\Ind{|X_{n,k}|\le \tau})$ which together with \eqref{eq:sum-of-rpobabilities} proves (ii).
   The convergence asserted in (i) is now an immediate consequence of (ii) and the convergence $\sum_{k=0}^{k_n} X_{n,k}\Ind{|X_{n,k}|\le \tau} \to 1$ in probability.
\end{proof}
\bibliographystyle{amsplain}
\bibliography{convergenceChaoses}

\end{document}